\documentclass[a4paper]{amsart}

\usepackage[colorlinks,linkcolor=blue,citecolor=magenta]{hyperref}
\usepackage{blkarray}
\usepackage{float,enumerate,cite}
\usepackage{pgf,tikz,pgfplots,marvosym}
\pgfplotsset{compat=1.15}
\usetikzlibrary{arrows}

\usepackage{graphics}

\usepackage{amsthm,amsmath,amssymb,mathtools,mathrsfs}

\allowdisplaybreaks

\newtheorem{thm}{Theorem}
\newtheorem{prop}[thm]{Proposition}
\newtheorem{lem}[thm]{Lemma}

\newtheorem{conj}{Conjecture}
\theoremstyle{definition}
\newtheorem{definition}{Definition}
\theoremstyle{remark}
\newtheorem{remark}{\bf Remark}
\newtheorem{example}{\bf Example}
\newtheorem{prob}{\bf Problem}
\title{Structure of Trees with Respect to Nodal Vertex Sets}

\author{Asghar Bahmani}
\address{Department of Mathematics and Computer Science\\
	Amirkabir University of Technology, Tehran, Iran\\
	{asghar.bahmani@aut.ac.ir}}

\author{Dariush Kiani}
\address{Department of Mathematics and Computer Science\\
Amirkabir University of Technology, Tehran, Iran \&
School of Mathematics,
Institute for Research in Fundamental Sciences (IPM), Tehran, Iran\\
{dkiani@aut.ac.ir}}

\begin{document}

\begin{abstract}
Let $T$ be a tree with a given adjacency  eigenvalue $\lambda$. In this paper, by using the $\lambda$-minimal trees, we determine the structure of trees with a given multiplicity of the eigenvalue $\lambda$. Furthermore, we consider the relationship between the structure of trees and the eigensystem of a given Laplacian eigenvalue.\\% This work is done in a new manner and proofs on some recent works in this topic.\\

\end{abstract}
\keywords{Tree, Eigenvector, Laplacian Matrix, Adjacency Matrix, Tree Structure}
\subjclass[2010]{ 05C05, 05C25, 05C50, 11R04, 15A18}
\maketitle

%\begin{keyword}
%%% keywords here, in the form: keyword \sep keyword  
%Tree,Eigenvector, Laplacian Matrix, Adjacency Matrix, Tree Structure
%%% MSC codes here, in the form: \MSC code \sep code
%\MSC[2010]  05C05, 05C25, 05C31, 05C50 11R04 15A18% 
%
%\end{keyword}

%\end{frontmatter}

%%
% Start line numbering here if you want
%\linenumbers
\section{Introduction}
All graphs in this paper are finite and simple, unless noted otherwise. We denote by $[n]$ the set $\{1,\ldots,n\}$.
For any integer $n\in\mathbb{N}$, we denote by $M_{n}(\mathbb{R})$ and $Sym_{n}(\mathbb{R})$ the set of all $n\times n$ real matrices and   all $n\times n$ symmetric matrices, respectively. The complete graph on $n$ vertices is denoted by $K_n$. 

The Laplacian (signless Laplacian, resp.) matrix of a simple graph $G$, denoted by $L_{G}$ ($Q_{G}$, resp.), is $D(G)-A_{G}$ ($D(G)+A_{G}$, resp.), where $D(G)$ is a diagonal matrix with diagonal entries $(d_{i})_i$ and $A_{G}$ is the adjacency matrix of $G$. 
A vertex $v\in V(G)$, and a subgraph $H$ of $G$, we denote by $N_{H}(v)$ the set $\{u: u\in V(H), uv\in E(G)\}$. The set of all adjacency eigenvalues of $G$ is denoted by $Spec(G)$.  

Let $T$ be a tree. We denote by $\rho(T)$ the spectral radius (the largest eigenvalue) of $T$.
We say a vector $\boldsymbol{x}\in\mathbb{R}^{n}$ is nowhere-zero, if for every $i\in [n]$, $\boldsymbol{x}_{i}\neq 0$.
For any integers $i$, $\bf{e}_{i}$ denotes the vector with a $1$ in the $i^{\rm th}$ coordinate and $0$'s elsewhere.
Let $\theta$ be an algebraic integer with the monic minimal polynomial $f(x)$. The \textit{norm} of $\theta$ is the product of the conjugates of $\theta$ and we denote it by ${\rm Norm}(\theta)$. We say $\theta$ is a \textit{totally real algebraic integer} if every root of $f(x)$ is real. We denote by $\mathtt{TRAI}$, the set of all totally real algebraic integers. We say $\theta$ is a \textit{totally positive algebraic integer} if every root of $f(x)$ is a positive real number. We denote by $\mathtt{TPAI}$, the set of all totally positive algebraic integers. The following theorem states that every element of $\mathtt{TRAI}$ is a tree eigenvalue.

\begin{thm}\cite[Theorem 1]{S1}\label{trai}
Every totally real algebraic integer is an eigenvalue of some finite tree.
\end{thm}

In this paper, we define the $\lambda$-minimal (resp., $\mu$-L-minimal for a Laplacian eigenvalue) trees and by using them, we determine the structure of trees with eigenvalue $\lambda$ (resp. $\mu$) of a given multiplicity. There are many papers on considering the structure of trees with a given adjacency eigenvalue $\lambda$. These papers consider the structure of trees and study the relationship between the structure of trees and their eigenvectors and eigenvalues. In these works, the authors consider the zero coordinates of eigenvectors and their corresponding vertices. See \cite{FGWG,Fi}, \cite{GF}, \cite{JLM,JLS,JLSS,MS}, and \cite{S2}.

In this paper, we consider the structure of acyclic matrices and prove results above in a unified and new manner.  Also, as a consequence, we obtain some results on the relationship between the structure of trees and the Laplacian eigenvectors.  
In Section 2, we consider the relationship between the multiplicity of an eigenvalue $\lambda$ of an acyclic matrix and its structure. In Sections 3, we consider the $\lambda$-minimal trees for the adjacency eigenvalues of trees. In Section 4, we define minimal cut-trees for the Laplacian eigenvalues.

\section{The Multiplicity and Nodal Vertex Sets}
Suppose that $M=[m_{ij}]\in Sym_{n}(\mathbb{R})$. We denote by $G_{M}$ the simple graph on $n$ vertices such that for every $i,j\in V(G_{M})$, $i$ and $j$ are adjacent if and only if $m_{ij}\neq 0$. Whenever $G_{M}$ is a tree we denote it by $T_{M}$. If $G_{M}$ is a forest (union of some trees), then we say $M$ is an \textit{acyclic} matrix. For every eigenvalue $\theta$ of $M$, the \textit{nodal vertex set} $\mathcal{N}_{\theta}(M)$ is the set 
$$\{u\in V(G_{M})\,:\,{\text{for every $\theta$-eigenvector}}\,\boldsymbol{x}, \boldsymbol{x}_{u}=0\}.$$
We use the following notations in the sequel:
\begin{itemize}
\item
$\mathcal{N}_{\theta}^{\circ}(M)=\{u\in V(G_{M})\,:\, u\in \mathcal{N}_{\theta}(M), N_{G_{M}}(u)\subseteq \mathcal{N}_{\theta}(M)\}$,
\item
$\partial \mathcal{N}_{\theta}(M)=\mathcal{N}_{\theta}(M)\setminus \mathcal{N}_{\theta}^{\circ}(M)$,
\item
$\mathcal{E}_{\theta}^{\circ}(M)=\{uv\in E(G_{M})\,:\, u,v\in \mathcal{N}_{\theta}(M)\}.$
\end{itemize}

\textbf{Structure of Acyclic Matrices:}
We need the following propositions and lemmas for the main theorem.
\begin{prop}\cite[Proposition 1]{Fi}\label{mul1}
	If $A$ is an  $n\times n$ acyclic irreducible matrix with an eigenvalue $\theta$ and a nowhere-zero $\theta$-eigenvector, then $m_{\theta}(A)=1$. Moreover, If $i\in [n]$ and  $B$ is obtained from $A$ by deleting the row  and column $i$, then $m_{\theta}(B)=0$.
\end{prop}

\begin{lem}\label{nonz}
Suppose that $A$ is an acyclic matrix. If $\theta$ is an eigenvalue of $A$, then there exists a $\theta$-eigenvector $\boldsymbol{x}$ such that for every  $u\notin \mathcal{N}_{\theta}(A)$, $\boldsymbol{x}_{u}\neq 0$.
\end{lem}
\begin{proof}
	Suppose that $m_{\theta}(A)=k$ and $\boldsymbol{x_1},\ldots,\boldsymbol{x_k}$ are $k$ orthogonal $\theta$-eigenvectors.
	If $\boldsymbol{\alpha} , \boldsymbol{\beta}\in\mathbb{R}^{n}$, then we can choose $c\in\mathbb{R}$ such that for the vector $\boldsymbol{\gamma}=\boldsymbol{\alpha}+c\boldsymbol{\beta}$, $\boldsymbol{\gamma}_i=0$ if and only if $\boldsymbol{\alpha}_i=\boldsymbol{\beta}_i=0$, $i\in [n]$. Therefore, we choose some real numbers $c_1,\ldots,c_{k}$ and put
	$\boldsymbol{x}=\sum_{i=1}^{k}c_i \boldsymbol{x_i}$ such that $\boldsymbol{x}_{u}\neq 0$ if $u\notin \mathcal{N}_{\theta}(A)$.

\end{proof}

\begin{lem}\cite[Lemma 9]{BK2}\label{mul01}
	Let $\mathcal{H}$ and $\mathcal{L}$ be two symmetric matrices with row and column indices $I$ and $J$, respectively. Suppose that $m_{\mathcal{H}}(\lambda)=1$ and $\mathcal{G}$ is the  symmetric matrix given below,
	\begin{center}
		$\mathcal{G}=
		\begin{blockarray}{cc|c|c}
		\begin{block}{c(c|c|c)}
		I  & \mathcal{H}   & \pmb{x} & A\\\cline{1-4}
		v & \pmb{x}^{T} &a & \pmb{y}   \\\cline{1-4}
		J & A^{T}& \pmb{y}^{T} &  \mathcal{L}  \\
		\end{block}
		\end{blockarray},$
	\end{center}
	for a matrix $A$ and a vector $\pmb{x}$. If $A^{T}\pmb{\alpha}=\pmb{0}$ and $\pmb{x}^{T}\pmb{\alpha}\neq 0$, for a $\lambda$-eigenvector $\pmb{\alpha}$ of $\mathcal{H}$, then $m_{\mathcal{G}-{\{v\}}}(\lambda)=m_{\mathcal{G}}(\lambda)+1$.
\end{lem}

\begin{remark}\label{linkzero}
	In the proof of \cite[Lemma 8]{BK2}, it is shown that for every $\lambda$-eigenvector $\bf{\xi}$ of $\mathcal{G}$, ${\bf{\xi}}_{v}=0$.
\end{remark}

\begin{prop}\cite[Corollary 10]{BK2}\label{adj11}
	Let  $H$ and $L$  be  two  vertex disjoint (weighted) graphs such that $m_{H}(\theta)=1$, for some $\theta\in \mathbb{R}$. Suppose that $G$ is a graph formed  by joining  an arbitrary vertex $v\in V(L)$ to some arbitrary vertices of $H$. If for a $\theta$-eigenvector $\boldsymbol{x}$ of $H$, $\sum_{u\in N_{H}(v)} \boldsymbol{x}_{u}\neq 0$, then $m_{L-v}(\theta)=m_{G}(\theta)$.
	\begin{figure}[H]
		\centering
\begin{tikzpicture}[scale=.6,line cap=round,line join=round,>=triangle 45,x=1.0cm,y=1.0cm]
\draw [line width=1.pt] (-3.5,0.84)-- (-5.2,1.54);
\draw [line width=1.pt] (-3.5,0.84)-- (-5.28,0.38);
\draw [line width=1.pt] (-5.44,0.94) circle (1.5cm);
\draw (-5.48,-1.26) node[anchor=north west] {$L$};
\draw (-1.66,-0.86) node[anchor=north west] {$H$};
\draw [line width=1.pt] (-1.38,0.85) circle (1.5cm);
\draw [line width=1.pt] (-3.5,0.84)-- (-1.78,1.44);
\draw [line width=1.pt] (-3.5,0.84)-- (-1.76,0.24);
\draw [line width=1.pt] (-5.04,0.94) circle (2.02cm);
\draw (-6.66,1.36) node[anchor=north west] {$L-v$};
\draw (-2.04,1.68) node[anchor=north west] {$\vdots$};
\draw (-5.16,1.78) node[anchor=north west] {$\vdots$};
\begin{scriptsize}
\draw [fill=red] (-3.5,0.84) circle (2.5pt);
\draw[color=blue] (-3.56,1.21) node {$v$};
\end{scriptsize}
\end{tikzpicture}
		\caption{The graph $G$ of Proposition \ref{adj11}. }
	\end{figure}
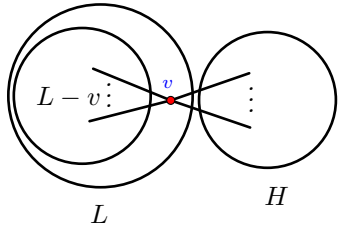
\end{prop}

The following elementary equation is important for us in the sequel. For any matrix $M=[m_{ij}]\in Sym_{n}(\mathbb{R})$ with an eigenvalue $\theta$ and a $\theta$-eigenvector $\boldsymbol{x}$, for every $i\in [n]$, we have
\begin{equation}\label{eigeneqn}
(\theta-m_{ii})\boldsymbol{x}_{i}=\sum_{j\neq i}m_{ij}\boldsymbol{x}_{j}=\sum_{ji\in E(G_{M})}m_{ij}\boldsymbol{x}_{j}.
\end{equation}

\textbf{Connection Between Two Matrices}: Let $A=[a_{ij}]\in Sym_{n}(\mathbb{R})$ and $B=[b_{ij}]\in Sym_{m}(\mathbb{R})$ for some $m,n\in\mathbb{N}$. We define a matrix $C$ as a connection between $A$ and $B$ for some $r\in[n], s\in[m]$ and a nonzero real number $\omega$, as follows:
\[
C=\begin{blockarray}{ccccccc}
\begin{block}{c(ccc|ccc)}
&&&&\bf{0}&\vdots&\bf{0}\\
r&&\bf{A}&&\vdots&\text{\small $\omega$}&\vdots\\
&&&&\bf{0}&\vdots&\bf{0}\\\cline{2-7}
&\bf{0}&\vdots&\bf{0}&&&\\
n+s&\vdots&\text{\small $\omega$}&\vdots&&\bf{B}&\\
&\bf{0}&\vdots&\bf{0}&&&\\
\end{block}
\end{blockarray}\,, \text{ where }
c_{ij}=\begin{cases}
a_{ij}& i,j\in [n],\\
b_{(i-n)(j-n)} & n+1\leq i,j\leq n+m,\\
\omega & i=r,j=s+n,\\
\omega & j=r,i=s+n,\\
0& \text{otherwise.}
\end{cases}
\]

\textbf{Acyclic Irreducible Matrix Construction:} Suppose that $\theta\in\mathbb{R}$. First, we define three subsets of symmetric matrices:
\begin{itemize}
	\item $\mathfrak{Z}_{\theta}=\{A: G_{A}\text{ is a tree,}\, m_{A}(\theta)=0\}\subset \bigcup_{n\geq 1}Sym_{n}(\mathbb{R})$,
	\item $\mathfrak{Min}_{\theta}=\{A: G_{A}\text{ is a tree and $A$ has a nowhere-zero $\theta$-eigenvector}\}\subset \bigcup_{n\geq 1}Sym_{n}(\mathbb{R})$,
	\item $\mathfrak{Link}_{\theta}=\{A=[a]_{1\times 1}: a\in\mathbb{R}\}$.
\end{itemize} 
By Proposition \ref{mul1} every element of $\mathfrak{Min}_{\theta}$ has the eigenvalue $\theta$ with the multiplicity 1.  
Now, we make an acyclic irreducible matrix by connecting some elements of $\mathfrak{Z}_{\theta}\bigcup\mathfrak{Min}_{\theta}\bigcup\mathfrak{Link}_{\theta}$ such that
\begin{enumerate}
	\item the resulting matrix is acyclic and irreducible,
	\item for any selected matrices  $\mathbf{m},\mathbf{m'}\in \mathfrak{Z}_{\theta}\bigcup\mathfrak{Min}_{\theta}$, we do not connect them,
	\item every chosen element of $\mathfrak{Link}_{\theta}$ (as a linking vertex) is connected to at least two elements of $\mathfrak{Min}_{\theta}$.
\end{enumerate}
We define the set $\mathcal{A}_{\theta}$ as\\
$
\mathcal{A}_{\theta}:=\{ P(\displaystyle\bigoplus_{i=1}^{k}A_i)P^{T}:\text{$P$ is a permutation matrix, $k\in\mathbb{N}$, and $A_i$ is obtained by the construction above} \},
$
where
\begin{center}
$
\displaystyle\bigoplus_{i=1}^{k}A_i :=\begin{bmatrix}
A_1& \bf{0} & \cdots &\bf{0}  \\ 
\bf{0} & A_2  &\cdots &\bf{0} \\ 
\vdots & \vdots & \ddots & \vdots \\ 
\bf{0} & \bf{0} &\cdots  & A_k
\end{bmatrix}.
$
\end{center}

In the following schematic figure, $\mathbf{m}_{i}\in \mathfrak{Z}_{\theta}\bigcup\mathfrak{Min}_{\theta}$, $i\in [n]$.

\begin{figure}[H]
\centering
\begin{tikzpicture}[scale=.7,line cap=round,line join=round,>=triangle 45,x=1.0cm,y=1.0cm]

\draw [line width=1.pt] (0.34,-1.08) circle (1.2cm);
\draw [line width=1.pt] (5.6,0.4) circle (1.2cm);
\draw [line width=1.pt] (9.46,-0.04) circle (1.2cm);
\draw [line width=1.pt] (1.26,-4.96) circle (1.2cm);
\draw [line width=1.pt] (-2.56,2.4) circle (1.2cm);
\draw [line width=1.pt] (8.82,-3.04) circle (1.2cm);
\draw [line width=1.pt] (6.08,-3.84) circle (1.2cm);
\draw [line width=1.pt] (5.98,3.98) circle (1.2cm);
\draw [line width=1.pt] (3.12,4.08) circle (1.2cm);
\draw [line width=1.pt] (9.14,3.13) circle (1.2cm);
\draw [line width=1.pt] (-3.44,-0.47) circle (1.2cm);
\draw [line width=1.pt] (-1.48,-5.09) circle (1.2cm);
\draw [line width=1.pt] (-3.24,-3.15) circle (1.2cm);
\draw [line width=1.pt] (-1.12,0.64)-- (-2.12,1.98);
\draw [line width=1.pt] (-1.12,0.64)-- (-2.82,-0.34);
\draw [line width=1.pt] (-1.12,0.64)-- (-0.02,-0.6);
\draw [line width=1.pt] (-0.46,-2.92)-- (-2.52,-2.94);
\draw [line width=1.pt] (-0.46,-2.92)-- (-1.04,-4.62);
\draw [line width=1.pt] (-0.46,-2.92)-- (0.86,-4.36);
\draw [line width=1.pt] (-0.46,-2.92)-- (0.14,-1.66);
\draw [line width=1.pt] (1.82,-0.74)-- (1.14,-0.94);
\draw [line width=1.pt] (1.82,-0.74)-- (2.44,-0.6);
\draw [line width=1.pt] (6.,-0.1)-- (6.7,-1.66);
\draw [line width=1.pt] (6.7,-1.66)-- (6.14,-3.54);
\draw [line width=1.pt] (6.7,-1.66)-- (8.34,-2.76);
\draw [line width=1.pt] (7.76,1.18)-- (6.46,0.84);
\draw [line width=1.pt] (7.76,1.18)-- (8.76,0.4);
\draw [line width=1.pt] (7.76,1.18)-- (8.48,2.42);
\draw [line width=1.pt] (4.78,2.46)-- (5.4,1.1);
\draw [line width=1.pt] (4.78,2.46)-- (5.84,3.4);
\draw [line width=1.pt] (4.78,2.46)-- (3.34,3.4);
\draw [line width=1.pt] (3.44,-0.36)-- (4.96,0.06);
\draw (2.4,-0.24) node[anchor=north west] {\LARGE $\cdots$};
\draw (-2.86,2.8) node[anchor=north west] {$\mathbf{m}_1$};
\draw (-3.94,-0.14) node[anchor=north west] {$\mathbf{m}_2$};
\draw (-3.7,-2.86) node[anchor=north west] {$\mathbf{m}_3$};
\draw (-1.8,-4.76) node[anchor=north west] {$\mathbf{m}_4$};
\draw (0.9,-4.56) node[anchor=north west] {$\mathbf{m}_5$};
\draw (0.0,-0.82) node[anchor=north west] {$\mathbf{m}_6$};

\draw (-0.56,3.99) node[anchor=north west] {$\mathbf{m}_7$};
\draw (1.0,1.72)  node[anchor=north west] {$\mathbf{m}_8$};

\draw (5.74,-3.64) node[anchor=north west] {$\mathbf{m}_n$};
\draw (8.26,-2.76) node[anchor=north west] {$\mathbf{m}_{n-1}$};
\draw (9.02,0.28) node[anchor=north west] {$\mathbf{m}_{n-2}$};
\draw (8.4,3.38) node[anchor=north west] {$\mathbf{m}_{n-3}$};
\draw (5.22,4.3) node[anchor=north west] {$\mathbf{m}_{n-4}$};
\draw (2.54,4.46) node[anchor=north west] {$\mathbf{m}_{n-5}$};
\draw (4.92,0.64) node[anchor=north west] {$\mathbf{m}_{n-6}$};

\draw [line width=1.pt] (-1.12,0.64)-- (-0.28,1.99);
\draw [line width=1.pt] (-0.28,1.99)-- (-0.28,3.34);
\draw [line width=1.pt] (-0.28,1.99)-- (0.58,1.56);
\draw [line width=1.pt] (-0.28,3.8) circle (1.2cm);
\draw [line width=1.pt] (1.28,1.52) circle (1.2cm);

\begin{scriptsize}

\draw[color=black] (-0.28,1.94)  node {\LARGE\textbullet};
\draw[color=black] (-0.28,1.99)  node {\Large$\bigodot$};
\draw[color=blue] (-0.21,1.6) node {\Large$\mathfrak{s}$};

\draw[color=black] (-1.12,0.60)  node {\LARGE\textbullet};
\draw[color=black] (-1.12,0.64)  node {\Large$\bigodot$};
\draw[color=blue] (-1.08,1.11) node {\Large$\mathfrak{u}$};
\draw [fill=red] (-0.46,-2.97) node {\LARGE\textbullet};
\draw [fill=red] (-0.46,-2.92)  node {\Large$\bigodot$};
\draw[color=blue] (-0.42,-2.45) node {\Large$\mathfrak{v}$};
\draw [fill=red] (1.82,-0.78)  node {\LARGE\textbullet};
\draw [fill=red] (1.82,-0.74)  node {\Large$\bigodot$};
\draw[color=blue] (1.90,-0.25) node {\Large$\mathfrak{w}$};
\draw [fill=red] (7.76,1.13) node {\LARGE\textbullet};
\draw [fill=red] (7.76,1.18)  node {\Large$\bigodot$};
\draw[color=blue] (7.8,1.67) node {\Large$\mathfrak{p}$};
\draw [fill=red] (6.7,-1.69) node {\LARGE\textbullet};
\draw [fill=red] (6.7,-1.66)  node {\Large$\bigodot$};
\draw[color=blue] (6.84,-1.19) node {\Large$\mathfrak{q}$};
\draw [fill=red] (4.78,2.41) node {\LARGE\textbullet};
\draw [fill=red] (4.78,2.46)  node {\Large$\bigodot$};
\draw[color=blue] (4.82,2.95) node {\Large$\mathfrak{r}$};
\end{scriptsize}
\end{tikzpicture}
\caption{A tree structure with linking vertices $\mathfrak{s},\mathfrak{u},\mathfrak{v},\mathfrak{w},\ldots,\mathfrak{p},\mathfrak{q},\mathfrak{r}$.}
\end{figure}
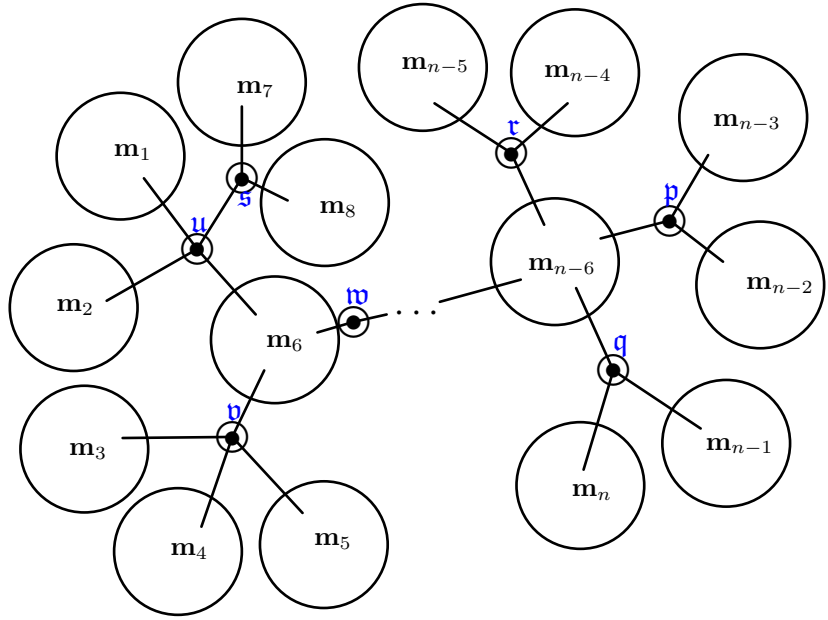

\begin{lem}\label{uniqrep}
	Every $\mathbf{m}\in \mathfrak{Min}_{\theta}$ has a unique representation (up to permutation) in $\mathcal{A}_{\theta}$.
\end{lem}
\begin{proof}
	By Proposition \ref{mul1}, $\mathbf{m}$ does not have any $\theta$-eigenvector with a zero entry. Suppose that $\bf{\xi}$ is a $\theta$-eigenvector of $\mathbf{m}$. There exists a pendant vertex of $T_{\mathbf{m}}$, say $u$, such that ${\bf{\xi}}_{u}\neq 0$. % (If there does not exist such vertex, we have $\bf{\xi}=\bf{0}$.)
	Now, we can choose the subtree $T^{\prime}$ of $T_{\mathbf{m}}$ containing $u$ such that $\mathbf{m}_{|T^{\prime}}\in \mathcal{A}_{\theta}$  and $T^{\prime}$ is adjacent with exactly one vertex of $T_{\mathbf{m}}-T^{\prime}$, say $v$. By Remark \ref{linkzero}, ${\bf{\xi}}_{v}=0$, a contradiction with Proposition \ref{mul1}. Hence, $\mathbf{m}$ has a unique representation in $\mathcal{A}_{\theta}$.
	
\end{proof}

The following theorem of Fiedler describes a relation between the multiplicity of an eigenvalue and the linking vertices.
\begin{thm}\cite[Theorem 2.4]{Fi}\label{mulnum}
	Let $\theta\in\mathbb{R}$ and $A=(a_{ik})$ be an $n\times n$ acyclic matrix, let $y=(y_{i})$ be a
	$\lambda$-eigenvector of $A$. If there are not two indices $i,k$ such that $a_{ik}\neq 0$ and $y_{i}=y_{k}=0$,
	then 
	\begin{center}
		$m_{\theta}(A)=c+\sum\limits_{k=3}^{n-1}(k-2)s_{k}$,
	\end{center}
	where $c$ is the number
	of components of $G_{A}$ and $s_{k}$,$(k=3,\ldots,n-1)$, is the number
	of those indices $j$ for which $y_{j}=0$ and $a_{jl}\neq 0$ for exactly $k$ indices $l\neq j$. In
	other words, $s_{k}$, is the number of vertices of $G_{A}$ corresponding to zero coordinates
	of $y$ and having degree $k$.
\end{thm}

We  generalize Theorem \ref{mulnum} for all cases of eigenvectors.
\begin{thm}\label{figen}
	Let $\theta\in \mathbb{R}$ and $A$ be an acyclic matrix. If $F=(V(G_{A})\setminus \mathcal{N}_{\theta}^{\circ}(A),E(G_{A})\setminus\mathcal{E}_{\theta}^{\circ}(A))$, then 
	
\begin{enumerate}[i.]
	\item $A\in \mathcal{A}_{\theta}$,
	\item \label{mulnumii} $\displaystyle m_{\theta}(A)=c+\displaystyle\sum\limits_{u\in \partial \mathcal{N}_{\theta}(A)}(d_{F}(u)-2)$, where $c$ is the number of components of $F$,
    \item $\displaystyle m_{\theta}(A)=|\{\text{componets of $(F-\partial \mathcal{N}_{\theta}(A))$ } \}|-|\partial \mathcal{N}_{\theta}(A)|$.
\end{enumerate}	

\end{thm}

\begin{proof} i,ii: Since the statements for $A$ and $PAP^{T}$, for any permutation matrix $P$, are equivalent, without loss of generality suppose that $A=\bigoplus_{i=1}^{k}A_i$ for some integer $k$ and irreducible acyclic matrices $A_1,\ldots,A_k$.
	
	First, if $m_{\theta}(A)=0$, then $A_i\in \mathfrak{Z}_{\theta}$, $i\in[k]$. Therefore, $A\in \mathcal{A}_{\theta}$, $F$ is empty, and $m_{\theta}(A)=0$.
	
	Now, suppose that $m_{\theta}(A)>0$ and by Lemma \ref{nonz}, $\boldsymbol{x}$ is a $\theta$-eigenvector such that for every  $u\notin \mathcal{N}_{\theta}(A)$, $\boldsymbol{x}_{u}\neq 0$.
	
	We prove theorem by induction on $|\partial \mathcal{N}_{\theta}(A)|$. If $|\partial \mathcal{N}_{\theta}(A)|=0$, then for every $i\in [k]$, $\boldsymbol{x}_{|A_i}=\boldsymbol{0}$ or $\boldsymbol{x}_{|A_i}$ is nowhere-zero. Assume that, without loss of generality, $m_{\theta}(A_1)=\cdots=m_{\theta}(A_t)=1$ and $m_{\theta}(A_{t+1})=\cdots=m_{\theta}(A_k)=0$, for some $t\in [k]$. Hence, $m_{\theta}(A)=t$, $A_1,\ldots,A_t \in\mathfrak{Min}_{\theta}$, and $A_{t+1},\ldots,A_k\in \mathfrak{Z}_{\theta}$.
	
    So, by Lemma \ref{uniqrep}, $A=\bigoplus_{i=1}^{k}A_i\in \mathcal{A}_{\theta}$, $F=G_{\bigoplus_{i=1}^{t}A_i}$, and $c=t$. Thus we have $m_{\theta}(A)=t=c+0=c+\displaystyle\sum\limits_{u\in \partial \mathcal{N}_{\theta}(A)}(d_{F}(u)-2)$.
    
   As the induction hypothesis, assume that the statements are true if $|\partial \mathcal{N}_{\theta}(A)|=r-1$ and we have $|\partial \mathcal{N}_{\theta}(A)|=r$.

    Suppose that,  without loss of generality, $m_{\theta}(A_1)>0$ and $|\partial \mathcal{N}_{\theta}(A_1)|>0$. Since $A_1$ is irreducible and acyclic, there exists an acyclic irreducible submatrix $B_1$ of $A_1$ such that $\boldsymbol{x}_{|B_1}$ is nowhere-zero and (from irreducibility of $A_1$) the vertices of $T_{B_1}$ have exactly one neighbor in $\partial \mathcal{N}_{\theta}(A_1)$, say $v$. Suppose that $C_1$ is the matrix that is obtained from $A_1$ by deleting the rows and columns corresponding to $B_1$ and $v$. By the induction hypothesis, $C_1\in \mathcal{A}_{\theta}$ and assume that $v_1,\ldots,v_s\in \mathfrak{Link}_{\theta}$, $\bf{u_1},\ldots,\bf{u_p}\in \mathfrak{Z}_{\theta}$, and $\bf{w_1},\ldots,\bf{w_q}\in \mathfrak{Min}_{\theta}$ are the submatrices of $C_1$ such that are connected to $v$. By (\ref{eigeneqn}), at least one of the neighbors of $v$, other than $B_1$, is an element of $\mathfrak{Min}_{\theta}$. Therefore, $A_1,A\in \mathcal{A}_{\theta}$.
   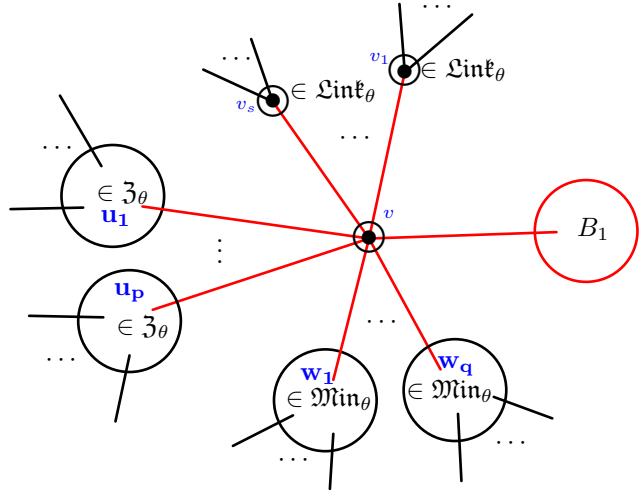
\begin{figure}[H]
 \centering
 \begin{tikzpicture}[scale=.5,line cap=round,line join=round,>=triangle 45,x=1.0cm,y=1.0cm]

 \draw [line width=1.pt] (15.11,0.66) circle (1.35cm);
 \draw [line width=1.pt] (20.70,-4.76) circle (1.35cm);
 \draw [color=red,line width=1.pt] (21.85,-0.47)-- (26.78,-0.30);
 \draw [color=red,line width=1.pt] (21.85,-0.47)-- (22.79,3.94);
 \draw [color=red,line width=1.pt] (21.85,-0.47)-- (19.32,3.16);  
 \draw (27.08,0.32) node[anchor=north west] {$B_1$};
 \draw (19.41,-4.18) node[anchor=north west] {$\in \mathfrak{Min}_{\theta}$};
% \draw (6.37,0.60) node[anchor=north west] {$A_{1}:$};
 \draw [color=red,line width=1.pt] (21.85,-0.47)-- (23.73,-3.94);
 \draw[color=blue] (24.13,-3.84) node {$\bf{w_q}$};  
 \draw [color=red,line width=1.pt] (21.85,-0.47)-- (20.91,-4.22);
 \draw[color=blue] (20.51,-4.12) node {$\bf{w_1}$};
 \draw [color=red,line width=1.pt] (21.85,-0.47)-- (15.89,0.38);
 \draw[color=blue] (15.19,0.08) node {$\bf{u_1}$}; 
 \draw [color=red,line width=1.pt] (21.85,-0.47)-- (16.16,-2.36);
 \draw[color=blue] (15.56,-1.9) node {$\bf{u_p}$}; 
 \draw [line width=1.pt] (13.72,3.32)-- (14.81,1.45);
 \draw [line width=1.pt] (14.33,0.35)-- (12.42,0.31);
 \draw [line width=1.pt] (14.86,-2.56)-- (12.91,-2.52);
 \draw [line width=1.pt] (15.55,-3.57)-- (15.18,-5.32);
 \draw [line width=1.pt] (19.85,-5.16)-- (18.27,-5.97);
 \draw [line width=1.pt] (20.91,-5.65)-- (20.82,-7.11);
 \draw [line width=1.pt] (24.19,-5.12)-- (24.28,-6.90);
 \draw [line width=1.pt] (25.13,-4.67)-- (26.67,-5.24);
 \draw [line width=1.pt] (22.79,3.94)-- (24.56,5.47);
 \draw [line width=1.pt] (22.79,3.94)-- (22.65,5.75);
 \draw [line width=1.pt] (19.32,3.16)-- (18.75,4.82);
 \draw [line width=1.pt] (19.32,3.16)-- (17.49,4.01);
 \draw (22.98,6.06) node[anchor=north west] {$\cdots$};
 \draw (24.92,-5.50) node[anchor=north west] {$\cdots$};
 \draw (19.20,-5.95) node[anchor=north west] {$\cdots$};
 \draw (13.09,-3.29) node[anchor=north west] {$\cdots$};
 \draw (12.99,2.34) node[anchor=north west] {$\cdots$};
 \draw (17.70,4.68) node[anchor=north west] {$\cdots$};
 \draw (21.52,-2.28) node[anchor=north west] {$\cdots$};
 \draw (17.58,0.18) node[anchor=north west] {$\vdots$};
 \draw (20.82,2.59) node[anchor=north west] {$\cdots$};
 \draw (22.59,-3.97) node[anchor=north west] {$\in \mathfrak{Min}_{\theta}$};
 \draw (14.43,1.29) node[anchor=north west] {$\in \mathfrak{Z}_{\theta}$};
 \draw (14.96,-2.27) node[anchor=north west] {$\in \mathfrak{Z}_{\theta}$};
 \draw [line width=1.pt] (15.55,-2.66) circle (1.35cm);
 \draw [line width=1.pt] (24.11,-4.49) circle (1.35cm);
 \draw [color=red,line width=1.pt] (27.56,-0.27) circle (1.35cm);
 \draw (22.94,4.54) node[anchor=north west] {$\in \mathfrak{Link}_{\theta}$};
 \draw (19.52,4.0) node[anchor=north west] {$\in \mathfrak{Link}_{\theta}$};
 \begin{scriptsize}
 \draw [fill=red] (21.85,-0.47) node {\LARGE\textbullet};
  \draw [fill=red] (21.85,-0.42) node {\Large$\bigodot$};
 \draw[color=blue] (22.37,0.17) node {$v$};
 \draw [fill=red] (22.79,3.94) node {\LARGE\textbullet};
  \draw [fill=red] (22.79,3.99) node {\Large$\bigodot$};
 \draw[color=blue] (22.14,4.31) node {${v_1}$};
 \draw [fill=red] (19.32,3.16) node {\LARGE\textbullet};
  \draw [fill=red] (19.32,3.19) node {\Large$\bigodot$};
 \draw[color=blue] (18.62,3.03) node {${v_s}$};
 \end{scriptsize}
 \end{tikzpicture}
 \caption{A schematic view of $A_1$ and the neighbors of $v$ in proof of Theorem \ref{figen}.}
   \end{figure}     
Assume that $F'=(V(G_{C_1})\setminus \mathcal{N}_{\theta}^{\circ}(C_1),E(G_{C_1})\setminus\mathcal{E}_{\theta}^{\circ}(C_1))$, $F_i=(V(G_{A_i})\setminus \mathcal{N}_{\theta}^{\circ}(A_i),E(G_{A_i})\setminus\mathcal{E}_{\theta}^{\circ}(A_i))$, and $c_i$ is the number of components of $F_i$, for $i\in [k]$. By Proposition \ref{adj11}, $m_{\theta}(A_{1})=m_{\theta}(C_1)$. Thus,
\begin{align*}
m_{\theta}(A)=\sum_{i=1}^{k}m_{\theta}(A_i)&=\sum_{i=2}^{k}m_{\theta}(A_i)+m_{\theta}(C_1)\\
&=\sum_{i=2}^{k}m_{\theta}(A_i)+(c_1+(q-1))+\sum_{ {u\in \partial \mathcal{N}_{\theta}(C_1)} }(d_{F'}(u)-2)\\
&=\sum_{i=2}^{k}m_{\theta}(A_i)+c_1+(d_{F_1}(v)-2)+\sum_{ \substack{u\in \partial \mathcal{N}_{\theta}(A_1)\\ u\neq v} }(d_{F_1}(u)-2)\\
&=\sum_{i=2}^{k}m_{\theta}(A_i)+c_1+\sum_{ \substack{u\in \partial \mathcal{N}_{\theta}(A_1)} }(d_{F_1}(u)-2)\\
&=\sum_{i=1}^{k}(c_i+\sum_{u\in \partial \mathcal{N}_{\theta}(A_i)}(d_{F_i}(u)-2))\\
&=c+\sum_{u\in \partial \mathcal{N}_{\theta}(A)}(d_{F}(u)-2).
\end{align*}
iii: If $m_{\theta}(A)=0$ or $|\partial \mathcal{N}_{\theta}(A)|=0$, then the statement is true. Suppose that $m_{\theta}(A)>0$, $|\{\text{componets of $(F-\partial \mathcal{N}_{\theta}(A))$ } \}|=p$, and $|\partial \mathcal{N}_{\theta}(A)|=q$ . Similar to the proof of i,ii, by Proposition \ref{adj11} and induction on $q$, we have $m_{\theta}(A_{1})=m_{\theta}(C_1)$ and
\begin{align*}
m_{\theta}(A)=\sum_{i=1}^{k}m_{\theta}(A_i)&=\sum_{i=2}^{k}m_{\theta}(A_i)+m_{\theta}(C_1)\\
&=(p-1)-(q-1)=p-q=|\{\text{componets of $(F-\partial \mathcal{N}_{\theta}(A))$ } \}|-|\partial \mathcal{N}_{\theta}(A)|.
\end{align*}
\end{proof}
\begin{remark}
The formula iii. in Theorem \ref{figen}, is a corollary of \cite[Theorem 2]{S2}.
\end{remark}
\begin{remark}
	Suppose that $T$ is a tree and $m_{T}(\lambda)=k$ for given $\lambda,k$. It can be easily seen that  $m_{T- \mathcal{N}^{\circ}_{\lambda}(A_{T})}(\lambda)=k$ and its $\lambda$-eigenvectors are  $\lambda$-eigenvectors of $T$  by removing the entries corresponding to $\mathcal{N}^{\circ}_{\lambda}(A_{T})$. %So, we can use Theorem \ref{mulnum} to prove relation \ref{mulnumii}. َ
	Also, in Theorem \ref{figen}, if we have the eigenvectors of $\mathbf{m}_{i}$, $i\in [n]$, we can obtain the eigenvectors of the resulting matrix easily.
\end{remark}

%\begin{remark}
%	By Theorem \ref{figen}, if we have the eigenvectors of $\mathbf{m}_{i}$, $i\in [n]$, then eigenvectors of resulting matrix will obtain easily.
%\end{remark}

\begin{example}
	For the tree $T$ that is shown in Figure \ref{rad2st}, we have $m_{T}(\sqrt{2})=3$ and it has the decomposition as below.
	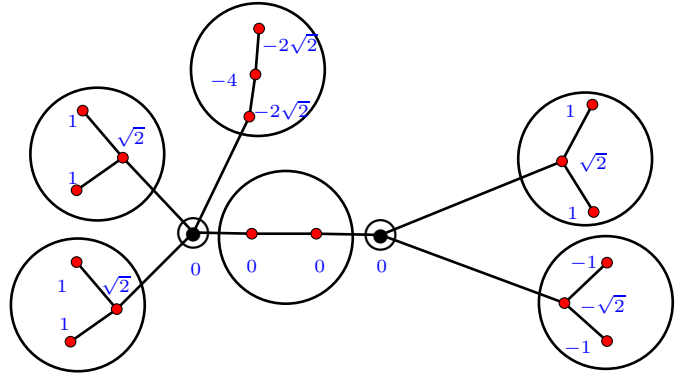
\begin{figure}[H]
		\centering
		\begin{tikzpicture}[scale=.8,line cap=round,line join=round,>=triangle 45,x=1.0cm,y=1.0cm]

		\draw [line width=1.pt] (-6.1,2.18)-- (-5.44,1.4);
		\draw [line width=1.pt] (-5.44,1.4)-- (-6.2,0.86);
		\draw [line width=1.pt] (2.28,2.28)-- (1.78,1.34);
		\draw [line width=1.pt] (1.78,1.34)-- (2.3,0.5);
		\draw [line width=1.pt] (-6.2,-0.33)-- (-5.54,-1.11);
		\draw [line width=1.pt] (-5.54,-1.11)-- (-6.3,-1.65);
		\draw [line width=1.pt] (2.52,-0.34)-- (1.82,-1.01);
		\draw [line width=1.pt] (1.82,-1.01)-- (2.52,-1.64);
		\draw [line width=1.pt] (-3.2,3.54)-- (-3.26,2.78);
		\draw [line width=1.pt] (-3.26,2.78)-- (-3.36,2.08);
		\draw [line width=1.pt] (-4.28,0.16)-- (-3.32,0.14);
		\draw [line width=1.pt] (-3.32,0.14)-- (-2.26,0.14);
		\draw [line width=1.pt] (-2.26,0.14)-- (-1.2,0.12);
		\draw [line width=1.pt] (-1.2,0.12)-- (1.78,1.34);
		\draw [line width=1.pt] (-1.2,0.12)-- (1.82,-1.01);
		\draw [line width=1.pt] (-4.28,0.16)-- (-5.44,1.4);
		\draw [line width=1.pt] (-4.28,0.16)-- (-5.54,-1.11);
		\draw [line width=1.pt] (-4.28,0.16)-- (-3.36,2.08);
		\draw [line width=1.pt] (-5.86,1.46) circle (1.1cm);
		\draw [line width=1.pt] (-6.18,-1.04) circle (1.1cm);
		\draw [line width=1.pt] (-3.22,2.86) circle (1.1cm);
		\draw [line width=1.pt] (-2.76,0.08) circle (1.1cm);
		\draw [line width=1.pt] (2.16,1.38) circle (1.1cm);
		\draw [line width=1.pt] (2.5,-1.) circle (1.1cm);
		\begin{scriptsize}
		\draw [fill=red] (-6.1,2.18) circle (2.5pt);
		\draw[color=blue] (-6.26,2.01) node {$1$};
		\draw [fill=red] (-5.44,1.4) circle (2.5pt);
		\draw[color=blue] (-5.32,1.75) node {$\sqrt{2}$};
		\draw [fill=red] (-6.2,0.86) circle (2.5pt);
		\draw[color=blue] (-6.26,1.07) node {$1$};
		\draw [fill=red] (2.28,2.28) circle (2.5pt);
		\draw[color=blue] (1.92,2.17) node {$1$};
		\draw [fill=red] (1.78,1.34) circle (2.5pt);
		\draw[color=blue] (2.28,1.33) node {$\sqrt{2}$};
		\draw [fill=red] (2.3,0.5) circle (2.5pt);
		\draw[color=blue] (1.96,0.49) node {$1$};
		\draw [fill=red] (-6.2,-0.33) circle (2.5pt);
		\draw[color=blue] (-6.44,-0.73) node {$1$};
		\draw [fill=red] (-5.54,-1.11) circle (2.5pt);
		\draw[color=blue] (-5.56,-0.71) node {$\sqrt{2}$};
		\draw [fill=red] (-6.3,-1.65) circle (2.5pt);
		\draw[color=blue] (-6.4,-1.35) node {$1$};
		\draw [fill=red] (2.52,-0.34) circle (2.5pt);
		\draw[color=blue] (2.14,-0.35) node {$-1$};
		\draw [fill=red] (1.82,-1.01) circle (2.5pt);
		\draw[color=blue] (2.44,-1.03) node {$-\sqrt{2}$};
		\draw [fill=red] (2.52,-1.64) circle (2.5pt);
		\draw[color=blue] (2.04,-1.75) node {$-1$};
		\draw [fill=red] (-3.2,3.54) circle (2.5pt);
		\draw[color=blue] (-2.7,3.29) node {$-2\sqrt{2}$};
		\draw [fill=red] (-3.26,2.78) circle (2.5pt);
		\draw[color=blue] (-3.78,2.67) node {$-4$};
		\draw [fill=red] (-3.36,2.08) circle (2.5pt);
		\draw[color=blue] (-2.84,2.19) node {$-2\sqrt{2}$};
		
		\draw [fill=red] (-4.28,0.11) node {\LARGE\textbullet};
		\draw [fill=red] (-4.28,0.16)  node {\Large$\bigodot$};
		\draw[color=blue] (-4.24,-0.45) node {$0$};
		\draw [fill=red] (-3.32,0.14) circle (2.5pt);
		\draw[color=blue] (-3.32,-0.39) node {$0$};
		\draw [fill=red] (-2.26,0.14) circle (2.5pt);
		\draw[color=blue] (-2.2,-0.39) node {$0$};
		
		\draw [fill=red] (-1.2,0.07) node {\LARGE\textbullet};
		\draw [fill=red] (-1.2,0.12)  node {\Large$\bigodot$};
		\draw[color=blue] (-1.18,-0.39) node {$0$};
		\end{scriptsize}
		\end{tikzpicture}
		\caption{The structure of a tree with a $\sqrt{2}$-eigenvector.}
		\label{rad2st}	
	\end{figure}	
	
\end{example}

\textbf{Parter-Weiner vertices:}
Suppose that $T$ is a tree with a given eigenvalue $\lambda$. The vertex $v$ of $T$ is called 
\begin{itemize}
	\item \textit{Parter-Wiener vertex} if  by removing $v$, the multiplicity of $\lambda$ increases by $1$,
	\item \textit{downer vertex} if  by removing $v$, the multiplicity of $\lambda$ decreases by $1$,
	\item \textit{neutral vertex} if  by removing $v$, the multiplicity of $\lambda$ does not change.
\end{itemize}
\begin{remark}
	It is easy to see that the linking vertices ($\partial \mathcal{N}_{\lambda}(A)$) in Theorem \ref{figen} are Parter-Weiner vertices. The vertices (indices) in $\mathcal{N}_{\lambda}^{\circ}(A)$ (the indices of elements of $\mathfrak{Z}_{\lambda}$) are neutral vertices and the vertices of the elements of $\mathfrak{Min}_{\lambda}$ are downer vertices. 
\end{remark}

\section{Minimal Trees with Respect to the Adjacency Eigenvalues}

In the sequel of the paper, for trees we denote by  $\mathcal{Z}_{\lambda}$ the set $\{T:T\text{ is a tree,}\, m_{T}(\lambda)=0\}$ and denote by $\mathcal{T}_{min,\lambda}$ the set $\{T: T\text{ is a tree with a nowhere-zero $\lambda$-eigenvector}\}$.

\begin{definition}\label{defmin}
	Let $\lambda\in \mathtt{TRAI}$. A tree $T$ is a \textit{$\lambda$-minimal} tree, if it has a nowhere-zero $\lambda$-eigenvector.
\end{definition}
Such trees in Definition \ref{defmin}, were defined in \cite{GL} as minimal trees and  in \cite{S2} as $\lambda$-primes.

\begin{remark}
	By Proposition \ref{mul1}, for a $\lambda$-minimal tree we have  $m_{T}(\lambda)=1$.
\end{remark}
\begin{lem}
	For every $\lambda\in\mathtt{TRAI}$, we have $\mathcal{T}_{min,\lambda}\neq \emptyset$.
\end{lem}
\begin{proof}
Suppose that $T$ is a tree such that $m_{T}(\lambda)>0$. By removing the vertices of $\mathcal{N}_{\lambda}(T)$ we have some components that have multiplicity 1 and they are minimal.
\end{proof}

By Perron–Frobenius theorem we have that every tree $T$ is an element of $\mathcal{T}_{min,\rho(T)}$ and we pose the following conjecture.
\begin{conj}
	Let $\theta\in\mathtt{TRAI}$ with the minimal polynomial $f(x)$. If $\lambda=\displaystyle \max_{f(z)=0} |z|$, then there exists a tree $T$ such that $\rho(T)=\lambda$.
\end{conj}

\begin{prop}\label{infmin}
	Let $\lambda\in \mathtt{TRAI}$. If $\lambda=0$, then $\mathcal{T}_{min,\lambda}=\{K_{1}\}$. If $\lambda\neq 0$, then  $|\mathcal{T}_{min,\lambda}|=+\infty$. 
\end{prop}

\begin{proof}
	First, suppose that $\lambda=0$. By (\ref{eigeneqn}), it is easy to see that the only $T\in \mathcal{T}_{min,0}$ is $K_1$.
	
	Now, suppose that  $\lambda\neq 0$. We construct from every element of $T\in \mathcal{T}_{min,\lambda}$, another element $T'\in \mathcal{T}_{min,\lambda}$ such that $|V(T)|<|V(T')|$. Assume that $T\in \mathcal{T}_{min,\lambda}$ with a pendant vertex $v$ and its neighbor $u$. Suppose that $\boldsymbol{x}$ is the $\lambda$-eigenvector such that $\boldsymbol{x}_{v}=1$. Hence, $\boldsymbol{x}_{u}=\lambda$.

	\begin{figure}[H]
		\centering
		\begin{tikzpicture}[scale=.8,line cap=round,line join=round,>=triangle 45,x=1.0cm,y=1.0cm]

		\draw [line width=1.pt] (-2.32,0.92)-- (-3.94,0.86);
		\draw [line width=1.pt] (-3.94,0.86)-- (-5.26,2.06);
		\draw [line width=1.pt] (-5.32,0.74)-- (-3.94,0.86);
		\draw [line width=1.pt] (-3.94,0.86)-- (-5.28,-0.22);
		\draw [line width=1.pt] (-4.86,0.9) circle (1.3cm);
		\draw (-8.,1.3) node[anchor=north west] {\Large $T:$};
		\draw (-2.58,0.7) node[anchor=north west] {$1$};
		\draw (-4.24,0.8) node[anchor=north west] {$\lambda$};
		\begin{scriptsize}
		\draw [fill=red] (-2.32,0.92) circle (2.5pt);
		\draw[color=blue] (-2.18,1.29) node {$v$};
		\draw [fill=red] (-3.94,0.86) circle (2.5pt);
		\draw[color=blue] (-3.9,1.23) node {$u$};
		\end{scriptsize}
		\end{tikzpicture}
		\caption{A $\lambda$-eigenvector of $T$.}
	\end{figure}
For an integer $k>1$, we construct the tree $T'$ as following: we use $2k-1$ copies $T_1,\ldots,T_k, T^{'}_{1},$ $\ldots,$ $T^{'}_{k-1}$ of $T$ and  identify all of the copies of vertex $v$, as shown below. Also, we connect to every copy $u^{'}_{i}$ of $u$, $i\in [k-1]$, two new pendant vertices. Now, we assign to $T_1,\ldots,T_k$ the vector $\boldsymbol{x}$ and to $T^{'}_{1},\ldots,T^{'}_{k-1}$ except $v$, the vector $-\boldsymbol{x}$, and assign $-1$ to new pendant vertices. By (\ref{eigeneqn}), it is easy to check that this vector is a nowhere-zero $\lambda$-eigenvector and hence, $T'\in \mathcal{T}_{min,\lambda}$.
	\begin{figure}[H]
		\centering
		\begin{tikzpicture}[scale=.8,line cap=round,line join=round,>=triangle 45,x=1.0cm,y=1.0cm]

		\draw [line width=1.pt] (3.56,1.12)-- (0.98,3.04);
		\draw [line width=1.pt] (0.98,3.04)-- (-0.34,4.24);
		\draw [line width=1.pt] (-0.4,2.92)-- (0.98,3.04);
		\draw [line width=1.pt] (0.98,3.04)-- (-0.36,1.96);
		\draw [line width=1.pt] (1.14,-0.68)-- (-0.34,0.33);
		\draw [line width=1.pt] (-0.4,-0.99)-- (1.14,-0.68);
		\draw [line width=1.pt] (1.14,-0.68)-- (-0.36,-1.95);
		\draw [line width=1.pt] (1.14,-0.68)-- (3.56,1.12);
		\draw [line width=1.pt] (3.56,1.12)-- (6.22,2.9);
		\draw [line width=1.pt] (3.56,1.12)-- (6.72,-0.58);
		\draw [line width=1.pt] (6.72,-0.58)-- (8.18,0.44);
		\draw [line width=1.pt] (6.72,-0.58)-- (8.28,-0.56);
		\draw [line width=1.pt] (6.72,-0.58)-- (8.16,-1.56);
		\draw [line width=1.pt] (6.22,2.9)-- (7.34,4.);
		\draw [line width=1.pt] (6.22,2.9)-- (7.6,2.98);
		\draw [line width=1.pt] (6.22,2.9)-- (7.58,2.02);
		\draw [line width=1.pt] (6.22,2.9)-- (5.7,3.92);
		\draw [line width=1.pt] (6.22,2.9)-- (5.18,3.32);
		\draw [line width=1.pt] (6.72,-0.58)-- (5.7,-1.52);
		\draw [line width=1.pt] (6.72,-0.58)-- (6.48,-1.82);
		\draw (5.82,1.64) node[anchor=north west] {$\vdots$};
		\draw (1.24,1.74) node[anchor=north west] {$\vdots$};
		\draw [line width=1.pt] (0.2,3.06) circle (1.4cm);
		\draw [line width=1.pt] (0.28,-0.82) circle (1.40cm);
		\draw [line width=1.pt] (7.06,2.94) circle (1.4cm);
		\draw [line width=1.pt] (7.64,-0.56) circle (1.4cm);
		\draw (-4,0.99) node[anchor=north west] {\Large $T':$};
		\draw (-1.92,3.66) node[anchor=north west] {$T_1$};
		\draw (-1.82,-0.5) node[anchor=north west] {$T_k$};
		\draw (8.9,3.68) node[anchor=north west] {$T_{1}^{'}$};
		\draw (9.16,-0.22) node[anchor=north west] {$T_{k-1}^{'}$};
		\draw (0.78,2.92) node[anchor=north west] {$\lambda$};
		\draw (1.0,-0.71) node[anchor=north west] {$\lambda$};
		\draw (3.33,0.86) node[anchor=north west] {$1$};
		\draw (5.82,2.78) node[anchor=north west] {$-\lambda$};
		\draw (6.5,-0.83) node[anchor=north west] {$-\lambda$};
		\draw (4.4,3.6) node[anchor=north west] {$-1$};
		\draw (5.2,4.54) node[anchor=north west] {$-1$};
		\draw (5.04,-1.6) node[anchor=north west] {$-1$};
		\draw (6.18,-1.92) node[anchor=north west] {$-1$};
		\begin{scriptsize}
		\draw [fill=red] (3.56,1.12) circle (2.5pt);
		\draw[color=blue] (3.56,1.57) node {$v$};
		\draw [fill=red] (0.98,3.04) circle (2.5pt);
		\draw[color=blue] (1.24,3.31) node {$u_1$};
		\draw [fill=red] (1.14,-0.68) circle (2.5pt);
		\draw[color=blue] (1.08,-0.27) node {$u_k$};
		\draw [fill=red] (6.22,2.9) circle (2.5pt);
		\draw[color=blue] (6.3,3.39) node {$u_{1}^{'}$};
		\draw [fill=red] (6.72,-0.58) circle (2.5pt);
		\draw[color=blue] (6.8,-0.01) node {$u_{k-1}^{'}$};
		\draw [fill=red] (5.7,3.92) circle (2.5pt);
		\draw [fill=red] (5.18,3.32) circle (2.5pt);
		\draw [fill=red] (5.7,-1.52) circle (2.5pt);
		\draw [fill=red] (6.48,-1.82) circle (2.5pt);
		\end{scriptsize}
		\end{tikzpicture}
		\caption{A $\lambda$-eigenvector of $T'$.}
	\end{figure}	
\end{proof}
In the following figure, there is another construction of a $\lambda$-minimal tree from $T$ of Proposition \ref{infmin}.

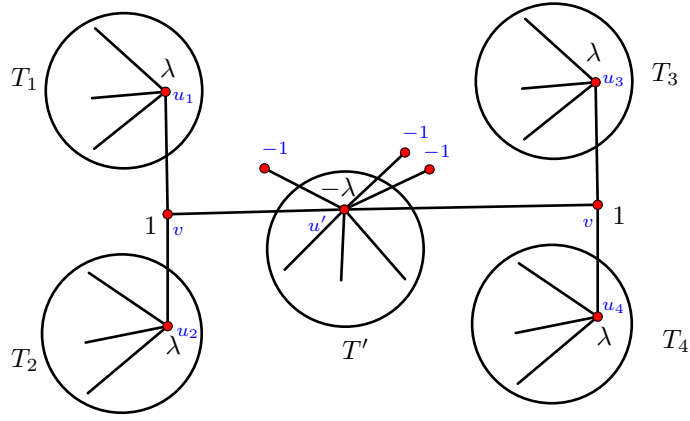
\begin{figure}[H]
	\centering
	\begin{tikzpicture}[scale=.7,line cap=round,line join=round,>=triangle 45,x=1.0cm,y=1.0cm]

	\draw [line width=1.pt] (10.9,1.65)-- (10.86,3.97);
	\draw [line width=1.pt] (10.86,3.97)-- (9.54,5.17);
	\draw [line width=1.pt] (9.48,3.85)-- (10.86,3.97);
	\draw [line width=1.pt] (10.86,3.97)-- (9.52,2.89);
	\draw [line width=1.pt] (10.9,-0.47)-- (9.42,0.54);
	\draw [line width=1.pt] (9.36,-0.78)-- (10.9,-0.47);
	\draw [line width=1.pt] (10.9,-0.47)-- (9.4,-1.74);
	\draw [line width=1.pt] (10.9,-0.47)-- (10.9,1.65);
	\draw [line width=1.pt] (10.08,3.99) circle (1.47cm);
	\draw [line width=1.pt] (10.04,-0.61) circle (1.5cm);
	\draw (11.74,4.52) node[anchor=north west] {$T_3$};
	\draw (11.94,-0.54) node[anchor=north west] {$T_4$};
	\draw (10.5,4.90) node[anchor=north west] {$\lambda$};
	\draw (10.7,-0.48) node[anchor=north west] {$\lambda$};
	\draw (11.,1.78) node[anchor=north west] {$1$};
	\draw [line width=1.pt] (2.82,1.47)-- (2.78,3.79);
	\draw [line width=1.pt] (2.78,3.79)-- (1.46,4.99);
	\draw [line width=1.pt] (1.4,3.67)-- (2.78,3.79);
	\draw [line width=1.pt] (2.78,3.79)-- (1.44,2.71);
	\draw [line width=1.pt] (2.82,-0.65)-- (1.34,0.36);
	\draw [line width=1.pt] (1.28,-0.96)-- (2.82,-0.65);
	\draw [line width=1.pt] (2.82,-0.65)-- (1.32,-1.92);
	\draw [line width=1.pt] (2.82,-0.65)-- (2.82,1.47);
	\draw [line width=1.pt] (2.,3.81) circle (1.47cm);
	\draw [line width=1.pt] (1.96,-0.79) circle (1.50cm);
	\draw (-0.3,4.42) node[anchor=north west] {$T_1$};
	\draw (-0.3,-0.9) node[anchor=north west] {$T_2$};
	\draw (2.5,4.6) node[anchor=north west] {$\lambda$};
	\draw (2.6,-0.66) node[anchor=north west] {$\lambda$};
	\draw (2.2,1.6) node[anchor=north west] {$1$};
	\draw [line width=1.pt] (2.82,1.47)-- (6.14,1.56);
	\draw [line width=1.pt] (6.14,1.56)-- (10.9,1.65);
	\draw [line width=1.pt] (6.16,0.81) circle (1.47cm);
	\draw (5.92,-0.72) node[anchor=north west] {$T'$};
	\draw (5.5,2.34) node[anchor=north west] {$-\lambda$};
	\draw [line width=1.pt] (6.14,1.56)-- (5.02,0.42);
	\draw [line width=1.pt] (6.14,1.56)-- (6.08,0.22);
	\draw [line width=1.pt] (6.14,1.56)-- (7.28,0.24);
	\draw [line width=1.pt] (6.14,1.56)-- (7.28,2.64);
	\draw [line width=1.pt] (6.14,1.56)-- (4.64,2.34);
	\draw [line width=1.pt] (6.14,1.56)-- (7.74,2.32);
	\begin{scriptsize}
	\draw [fill=red] (10.9,1.65) circle (2.5pt);
	\draw[color=blue] (10.72,1.3) node {$v$};
	\draw [fill=red] (10.86,3.97) circle (2.5pt);
	\draw[color=blue] (11.2,4.) node {$u_3$};
	\draw [fill=red] (10.9,-0.47) circle (2.5pt);
	\draw[color=blue] (11.2,-0.3) node {$u_4$};
	\draw [fill=red] (2.82,1.47) circle (2.5pt);
	\draw[color=blue] (3.02,1.15) node {$v$};
	\draw [fill=red] (2.78,3.79) circle (2.5pt);
	\draw[color=blue] (3.13,3.68) node {$u_1$};
	\draw [fill=red] (2.82,-0.65) circle (2.5pt);
	\draw[color=blue] (3.2,-0.78) node {$u_2$};
	\draw [fill=red] (6.14,1.56) circle (2.5pt);
	\draw[color=blue] (5.64,1.3) node {$u'$};
	\draw [fill=red] (7.28,2.64) circle (2.5pt);
	\draw[color=blue] (7.48,3.01) node {$-1$};
	\draw [fill=red] (4.64,2.34) circle (2.5pt);
	\draw[color=blue] (4.84,2.71) node {$-1$};
	\draw [fill=red] (7.74,2.32) circle (2.5pt);
	\draw[color=blue] (7.88,2.65) node {$-1$};
	\end{scriptsize}
	\end{tikzpicture}
	\caption{A $\lambda$-minimal tree and a $\lambda$-eigenvector.}
\end{figure}

Suppose that $k\in\mathbb{N}$ and $\lambda_{1},\ldots,\lambda_{k}\in\mathtt{TRAI}$. It is easy to make a tree $T$ such that  
	$\lambda_{1},\ldots,\lambda_{k}\in\,Spec(T)$. We pose the following conjecture about the minimal trees of given eigenvalues. 

\begin{conj}
	Suppose that $k\in\mathbb{N}$ and $\lambda_{1},\ldots,\lambda_{k}\in\mathtt{TRAI}\setminus\{0\}$. Then
	\begin{itemize}
		\item $\displaystyle \mathcal{T}_{min,\lambda_{1}}\nsubseteq \displaystyle \mathcal{T}_{min,\lambda_{2}}$,
		\item $\displaystyle\bigcap_{i=1}^{k}  \displaystyle \mathcal{T}_{min,\lambda_{i}}\neq \emptyset$.
	\end{itemize} 
\end{conj}

\section{Minimal Trees with Respect to the Laplacian Eigenvalues}
In this section, we consider the structure of trees with respect to the Laplacian eigenvectors. First, we recall and generalize two theorems on the integer Laplacian eigenvalues.
\begin{prop}\cite[Proposition 6]{GM}\label{integer}
	Let $G$ be a graph with $n$ vertices and $t\geq 1$ spanning trees. If $\mu$ is a positive integer eigenvalue of $L_{G}$, then $\mu |nt$. If $G$ is Laplacian integral, then $\mu^{m}|nt$, where $m=m_{L_{G}}(\mu)$.
\end{prop}
A generalization of Proposition \ref{integer} is the following proposition. 
\begin{prop}\label{genintegergraph}
	Let $G$ be a connected graph on $n$ vertices and $t$ spanning trees. If $\mu$ is a nonzero eigenvalue of $L_{G}$, then $({\rm Norm}(\mu))^{m}|nt$, where $m=m_{L_{G}}(\mu)$.
\end{prop}
\begin{proof}
	Suppose that $\psi(x)$ is the characteristic polynomial of $L_{G}$. So,  $\psi(x)=x\prod_{i}f_{i}(x)^{m_i}=x(x^{n-1}-\cdots +(-1)^{n-1}nt)$, where $\{f_{i}(x)\}$ are the minimal polynomials of the nonzero eigenvalues of $L_{G}$. Hence $\prod_{i}f_{i}(0)^{m_i}=nt$. Therefore, $({\rm Norm}(\mu))^{m}|nt$.
\end{proof}
Another well-known theorem on the multiplicity of integer Laplacian eigenvalues of trees is the following theorem.
\begin{thm}\cite[Theorem 2.1]{GMS}\label{integertree}
	Suppose $T$ is a tree on $n$ vertices. If  $\mu>1$ is an integer eigenvalue of $L_{T}$ with corresponding eigenvector $u$, then
	
	\begin{enumerate}[i.]
		\item
		$\mu|n$;\hspace*{1cm}
		\item
		no coordinate of $u$ is zero;\hspace*{1cm}
		\item
		$m_{L_{T}}(\mu)=1$.
	\end{enumerate}
\end{thm}

We prove the following generalization of Theorem \ref{integertree}.

\begin{thm}\label{genintegertree}
	Let $T$ be a tree on $n$ vertices. If $\mu$ is an eigenvalue of $L_{T}$ and $ Norm(\mu)>1$, then
	
	\begin{enumerate}[i.]
		\item
		${\rm Norm}(\mu)|n$;\hspace*{1cm}
		\item
		no coordinate of a $\mu$-eigenvector is zero;\hspace*{1cm}
		\item
		$m_{L_{T}}(\mu)=1$.
	\end{enumerate}
\end{thm}
\begin{proof}
i. It is a corollary of Proposition \ref{genintegergraph}.\\
ii. Suppose by contradiction that $\bf{\xi}$ is a $\mu$-eigenvector with some zero coordinates. By deleting the rows and columns corresponding to the zero entries of $\bf{\xi}$, we have at least one submatrix $L$ of $L_{T}$ that 
\[
L=L_{T'}+\bf{e}_{i}\bf{e}_{i}^{T},
\]
for a subtree $T'$ and an index $i$, such that $L$ has a nowhere-zero $\mu$-eigenvector. Therefore ${\rm Norm}(\mu)|det(L)$. But $det(L)=det(L_{T'})+{\bf{e}}_{i}^{T}adj(L_{T'}){\bf{e}}_{i}=1$ (see \cite[Problem 3.1]{Pr}) and hence this is a contradiction.\\
iii. By previous item and Proposition \ref{mul1}, we have $m_{L_{T}}(\mu)=1$.
\end{proof}

%----------------------------------------------------------------

\begin{definition}
	A \textit{half-edge} is an edge with one end-vertex. A \textbf{cut-tree} is obtained by adding some half-edges to a tree. For an integer $k\geq 0$, a $k$-cut-tree is a cut-tree with $k$ half-edges. 
\end{definition}

\begin{example}
Every tree is a $0$-cut-tree. In the figure below, there are a $1$-cut-tree and a $4$-cut-tree.
	\begin{figure}[H]
		\centering
		\begin{tikzpicture}[scale=1.2,line cap=round,line join=round,>=triangle 45,x=1.0cm,y=1.0cm]

		\draw [line width=1.pt] (0.46,0.4)-- (-0.14,0.98);
		\draw [line width=1.pt] (2.1,0.4)-- (0.46,0.4);
		\draw [line width=1.pt] (0.46,0.4)-- (-0.14,-0.04);
		\draw (1.46,-0.38) node[anchor=north west] {$\mathfrak{T}_c$: $4$-cut-tree};
		\draw [line width=1.pt] (2.1,0.4)-- (3.92,0.4);
		\draw [line width=1.pt] (3.92,0.4)-- (4.6,0.98);
		\draw [line width=1.pt] (3.92,0.4)-- (4.6,-0.18);
		\draw [line width=1.pt] (-2.58,0.31)-- (-3.54,0.31);
		\draw (-3.7,-0.42) node[anchor=north west] {$\mathfrak{T}'_c$: $1$-cut-tree};
		\begin{scriptsize}
		\draw [fill=red] (0.46,0.4) circle (2.5pt);
		%	\draw[color=blue] (0.6,0.6) node {$u$};
		\draw[] (-0.14,0.98) node {\Large\LeftScissors};
		\draw[] (-0.14,-0.04) node {\Large\LeftScissors};		
		\draw [fill=red] (2.1,0.4) circle (2.5pt);
		%	\draw[color=blue] (2.24,0.6) node {$v$};
		\draw [fill=red] (3.92,0.4) circle (2.5pt);
		%	\draw[color=blue] (3.86,0.6) node {$w$};
		\draw[] (4.6,-0.18) node {\Large\RightScissors};
		\draw[] (4.6,0.98) node {\Large\RightScissors};		
		\draw [fill=red] (-2.58,0.31) circle (2.5pt);
		%	\draw[color=blue] (-2.56,0.6) node {$o$};
		\draw[] (-3.54,0.31) node {\Large\LeftScissors};
		\end{scriptsize}
		\end{tikzpicture}
		\caption{Two cut-trees.}
	\end{figure}
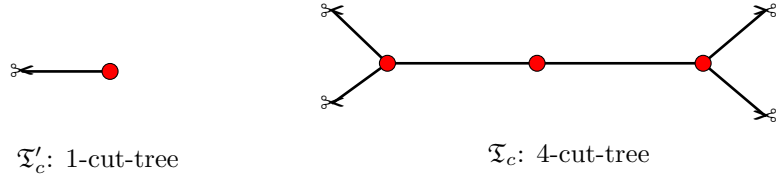	
	The adjacency matrix and the Laplacian matrix of the cut-trees above are
	\[
	A_{\mathfrak{T}'_c}=\begin{bmatrix}
	0
	\end{bmatrix},\,
	L_{\mathfrak{T}'_c}=\begin{bmatrix}
	1
	\end{bmatrix},\,
	A_{\mathfrak{T}_c}=\begin{bmatrix}
	0&1  & 0 \\ 
	1& 0 & 1 \\ 
	0& 1 & 0
	\end{bmatrix},\,
	L_{\mathfrak{T}_c}=\begin{bmatrix}
	3& -1  & 0 \\ 
	-1& 2 & -1 \\ 
	0& -1 & 3
	\end{bmatrix}. 
	\]	 
\end{example}

\begin{definition}
	Let $\mu\in \mathtt{TPAI}$. A cut-tree $\mathfrak{T}_{c}$ is a $\mu$-L-minimal cut-tree, if $L_{\mathfrak{T}_{c}}$ has a nowhere-zero $\mu$-eigenvector $\boldsymbol{x}$ such that for every $uv\in E(\mathfrak{T}_{c})$, $\boldsymbol{x}_{u}\neq \boldsymbol{x}_{v}$.
\end{definition}
We denote by $\mathcal{LT}_{min,\mu}^{k}$  the set of all $\mu$-L-minimal $k$-cut-trees and $\mathcal{LT}_{min,\mu}:=\displaystyle\bigcup_{k=0}^{+\infty}\mathcal{LT}_{min,\mu}^{k}$. For brevity, we use $\mu$-L-minimal tree instead of $\mu$-L-minimal $0$-cut-tree.
\begin{remark}
	For every $\mu$-L-minimal cut-tree $\mathfrak{T}_{c}$ and for every $uv\in E(\mathfrak{T}_{c})$, we have $m_{L_{\mathfrak{T}_{c}-uv}}(\mu)=0$.
\end{remark}

\begin{lem}
	Let $\mu$ be a Laplacian eigenvalue of a tree $T$. If a $\mu$-eigenvector $\bf{\xi}$ vanishes at a vertex of $T$, then $\mathcal{LT}_{min,\mu}^{1} \neq \emptyset$.
\end{lem}
\begin{proof}
Since $L_{T}$ is an acyclic matrix, by removing the indices of $L_{T}$ that are corresponding to the zero entries of $\bf{\xi}$, we obtain at least two Laplacian matrices of two elements of  $\mathcal{LT}_{min,\mu}^{1}$.
\end{proof}

\begin{lem}
	Let $s$ be a positive Laplacian eigenvalue of a tree $T$. If $s>1$, then $|\mathcal{LT}_{min,s}^{0}|= +\infty$.
\end{lem}
\begin{proof}
By Theorem \ref{integertree}, $T$ is an $s$-L-minimal tree. Suppose that $v$ is a pendant vertex of $T$ with its neighbor $u$. Assume that $\bf{\xi}$ is the  $s$-eigenvector that  ${\bf{\xi}}_{v}=1$ and ${\bf{\xi}}_{u}=1-s$. By (\ref{eigeneqn}), it is easy to see that the following construction, as shown in figure below, is a new $s$-L-minimal tree.
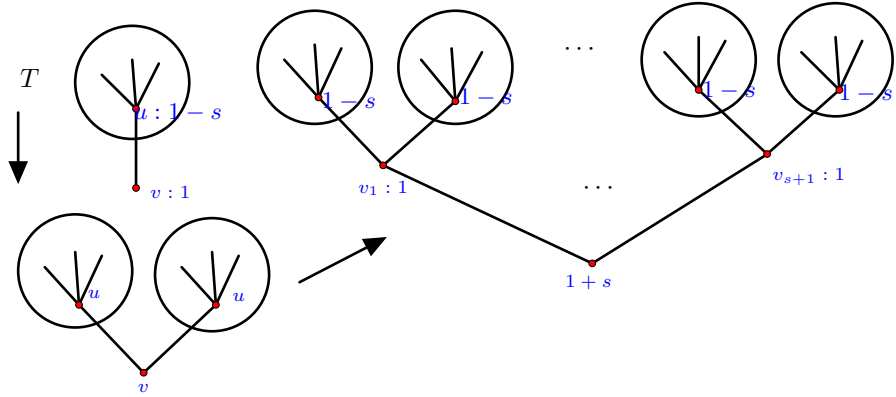
\begin{figure}[H]
	\centering
	\begin{tikzpicture}[scale=.5,line cap=round,line join=round,>=triangle 45,x=1.0cm,y=1.0cm]
	\draw [line width=1.pt] (-5.2,0.3)-- (-4.6,1.6);
	\draw [line width=1.pt] (-5.3,1.7)-- (-5.2,0.3);
	\draw [line width=1.pt] (-5.2,0.3)-- (-6.1,1.3);
	\draw [line width=1.pt] (-5.3,1.1) circle (1.5cm);
	\draw [line width=1.pt] (-3.5,-1.5)-- (-1.6,0.2);
	\draw [line width=1.pt] (-1.6,0.2)-- (-0.9,1.5);
	\draw [line width=1.pt] (-1.7,1.6)-- (-1.6,0.2);
	\draw [line width=1.pt] (-1.6,0.2)-- (-2.5,1.3);
	\draw [line width=1.pt] (-1.6,1.1) circle (1.5cm);
	\draw [line width=1.pt] (-5.2,0.3)-- (-3.5,-1.5);
	\draw [line width=1.pt] (4.8,0.5)-- (5.5,1.8);
	\draw [line width=1.pt] (4.8,1.9)-- (4.8,0.5);
	\draw [line width=1.pt] (4.8,0.5)-- (3.9,1.5);
	\draw [line width=1.pt] (4.8,1.3) circle (1.5cm);
	\draw [line width=1.pt] (6.6,-1.2)-- (8.5,0.5);
	\draw [line width=1.pt] (8.5,0.5)-- (9.1,1.8);
	\draw [line width=1.pt] (8.4,1.9)-- (8.5,0.5);
	\draw [line width=1.pt] (8.5,0.5)-- (7.5,1.5);
	\draw [line width=1.pt] (8.4,1.3) circle (1.5cm);
	\draw [line width=1.pt] (4.8,0.5)-- (6.6,-1.2);
	\draw [line width=1.pt] (-3.5,-1.5)-- (2.0,-4.1);
	\draw [line width=1.pt] (2.0,-4.1)-- (6.61,-1.2);
	\draw (1.5,-1.7) node[anchor=north west] {$\cdots$};
	\draw [line width=1.pt] (-10.0,-0.0)-- (-9.4,1.2);
	\draw [line width=1.pt] (-10.1,1.3)-- (-10.0,-0.0);
	\draw [line width=1.pt] (-10.0,-0.0)-- (-10.9,0.9);
	\draw [line width=1.pt] (-10.1,0.7) circle (1.5cm);
	\draw [line width=1.pt] (-10.0,-0.0)-- (-10.0,-2.1);
	\draw (1.0,2.0) node[anchor=north west] {$\cdots$};
	\draw [line width=1.pt] (-11.5,-5.2)-- (-10.9,-3.9);
	\draw [line width=1.pt] (-11.6,-3.8)-- (-11.5,-5.2);
	\draw [line width=1.pt] (-11.5,-5.2)-- (-12.4,-4.2);
	\draw [line width=1.pt] (-11.6,-4.3) circle (1.5cm);
	\draw [line width=1.pt] (-9.8,-7.0)-- (-7.9,-5.2);
	\draw [line width=1.pt] (-7.9,-5.2)-- (-7.3,-3.9);
	\draw [line width=1.pt] (-8.0,-3.8)-- (-7.9,-5.2);
	\draw [line width=1.pt] (-7.9,-5.2)-- (-8.8,-4.2);
	\draw [line width=1.pt] (-8.0,-4.4) circle (1.5cm);
	\draw [line width=1.pt] (-11.5,-5.2)-- (-9.8,-7.0);
	\draw (-13.3,1.3) node[anchor=north west] {$T$};
	\draw [->,line width=1.pt] (-13.1,-0.1) -- (-13.1,-2.0);
	\draw [->,line width=1.pt] (-5.7,-4.6) -- (-3.4,-3.4);
	\begin{scriptsize}
	\draw [fill=red] (-5.2,0.3) circle (2.5pt);
	\draw[color=blue] (-4.4,0.2) node {\small $1-s$};
	\draw [fill=red] (-3.5,-1.5) circle (2.5pt);
	\draw[color=blue] (-3.5,-2.1) node {$v_1 :1$};
	\draw [fill=red] (-1.6,0.2) circle (2.5pt);
	\draw[color=blue] (-0.8,0.3) node {\small $1-s$};
	\draw [fill=red] (4.8,0.5) circle (2.5pt);
	\draw[color=blue] (5.6,0.5) node {\small $1-s$};
	\draw [fill=red] (6.6,-1.2) circle (2.5pt);
	\draw[color=blue] (7.7,-1.8) node {$v_{s+1} :1$};
	\draw [fill=red] (8.5,0.5) circle (2.5pt);
	\draw[color=blue] (9.2,0.4) node {\small $1-s$};
	\draw [fill=red] (2.0,-4.1) circle (2.5pt);
	\draw[color=blue] (1.9,-4.6) node {$1+s$};
	\draw [fill=red] (-10.0,-0.0) circle (2.5pt);
	\draw[color=blue] (-8.9,-0.1) node {\small $u:1-s$};
	\draw [fill=red] (-10.0,-2.1) circle (2.5pt);
	\draw[color=blue] (-9.1,-2.2) node {$v:1$};
	\draw [fill=red] (-11.5,-5.2) circle (2.5pt);
	\draw[color=blue] (-11.1,-4.9) node {$u$};
	\draw [fill=red] (-9.8,-7.0) circle (2.5pt);
	\draw[color=blue] (-9.8,-7.4) node {$v$};
	\draw [fill=red] (-7.9,-5.2) circle (2.5pt);
	\draw[color=blue] (-7.3,-5.0) node {$u$};
	\end{scriptsize}
	\end{tikzpicture}
	\caption{An $s$-L-minimal tree and an $s$-eigenvector.}
\end{figure}
\end{proof}

If ${\rm Norm}(\mu)>1$, then $\mathcal{LT}_{min,\mu}^{1}=\emptyset$, since $|det(L_{\mathfrak{T}_{c}})|=1$, for every $\mathfrak{T}_{c}\in \mathcal{LT}_{min,\mu}^{1}$.

\begin{conj}
	Let $\mu$ be a Laplacian eigenvalue of some tree $T$ and $\mu\neq 1$. Then
	\begin{itemize}
		\item If ${\rm Norm}(\mu)=1$, then $|\mathcal{LT}_{min,\mu}^{0}|=|\mathcal{LT}_{min,\mu}^{1}|=+\infty$.
		\item If ${\rm Norm}(\mu)>1$, then $|\mathcal{LT}_{min,\mu}^{0}|=+\infty$.	
	\end{itemize}
	
\end{conj}

Similar to Theorem \ref{trai}, we state the following conjecture.
\begin{conj}\label{tpai}
	Every totally positive algebraic integer is a Laplacian eigenvalue of some tree.
\end{conj}

For bipartite graphs (and therefore for trees), we know the Laplacian spectrum and the signless Laplacian spectrum are equal (see \cite[Proposition 1.3.10]{BH}). 

The subdivision of an edge $e=uv$ obtained by deleting the edge $e$ and adding a new vertex $w$ and two new edges $wu$ and $wv$. The subdivision of a graph $G$ is obtained  by the subdivision of all edges of $G$ and we denote by $S(G)$.
We have the following relation between the  signless Laplacian eigenvalues and the adjacency eigenvalues of graphs.
\begin{lem}\cite{CRS}\label{subdsign}
	For any graph $G$ with $m$ edges and $n$ vertices, we have
\[	
\lambda^{m-n}\phi_{Q_{G}}(\lambda^{2})=\phi_{S(G)}(\lambda),
\]
where $\phi_{Q_{G}}$ is the characteristic polynomial of the signless Laplacian matrix of $G$ and $\phi_{S(G)}$ is the characteristic polynomial of the adjacency matrix of $G$.
\end{lem}	 
Therefore, by Lemma \ref{subdsign}, Conjecture \ref{tpai} is equivalent to the following statement:

\begin{center}
\textbf{Every totally real algebraic integer is an eigenvalue of the subdivision of some finite tree.}
\end{center}

\subsection{Laplacian Tree Structure and Eigenvectors}

\begin{definition}
	Let $\mu\in \mathtt{TPAI}$. A cut-tree $\mathfrak{T}_{c}$ is a $\mu$-L cut-tree, if $L_{\mathfrak{T}_{c}}$ has a nowhere-zero $\mu$-eigenvector.
\end{definition}
We denote by $\mathcal{LT}_{\mu}^{k}$  the set of all $\mu$-L $k$-cut-trees and $\mathcal{LT}_{\mu}:=\displaystyle\bigcup_{k=0}^{+\infty}\mathcal{LT}_{\mu}^{k}$.
\begin{remark}\label{mincut}\textbf{Laplacian Tree Structure.}
	For every $\mu$-L cut-tree $\mathfrak{T}_{c}$ and  $\mu$-eigenvector $\boldsymbol{x}$, if we remove the edges $\{uv\in E(\mathfrak{T}_{c})\,:\,\boldsymbol{x}_{u}=\boldsymbol{x}_{v}\}$, we have some $\mu$-L-minimal cut-trees \cite[see Lemma 8]{BK}. Also, if we join by an edge a vertex of a $\mu$-L cut-tree and a vertex of another $\mu$-L cut-tree, we will obtain a new $\mu$-L cut-tree.

Suppose that $T$ is a tree, $\mu\in \mathbb{R}$, and $m_{L_{T}}(\mu)=k$. By Theorem \ref{figen}, if we remove $\partial \mathcal{N}_{\mu}(L_{T})$ from $L_{T}$ we have Laplacian matrices of some $\mu$-L cut-trees such as  $L_{\mathfrak{T}_{c}}\in \mathfrak{Min}_{\mu}$ and some cut-trees such as $\mathfrak{T'}_{c}$ where $L_{\mathfrak{T'}_{c}}\in \mathfrak{Z}_{\mu}$

\end{remark}

\begin{example}
Suppose that $k\in\mathbb{N}$, $\lambda\in\mathtt{TRAI}$, and $\mu\in\mathtt{TPAI}$. We characterize all trees $T$ such that they have the minimum vertices and
\begin{itemize}
	\item $m_{T}(\lambda)=k$;
	\item $m_{L_{T}}(\mu)=k$.
\end{itemize} 
If $k=1$, then every $\lambda$-minimal tree and every $\mu$-L-minimal tree with minimum vertices is a solution. 
If $k>1$, by Theorem \ref{figen}, a solution is a tree $T$ with minimum vertices has only one linking vertex and $k+1$ $\lambda$-minimal trees  (and  $k+1$ $\mu$-L-minimal cut-trees) with minimum vertices that connected to the linking vertex. 

\begin{figure}[H]
\centering
\begin{tikzpicture}[scale=.8,line cap=round,line join=round,>=triangle 45,x=1.0cm,y=1.0cm]

\draw [line width=1.pt,color=red] (-2.38,0.2)-- (-2.72,-0.7);
\draw [line width=1.pt,color=red] (-6.38,0.2)-- (-6.72,-0.7);
\draw [line width=1.pt,color=black] (-6.38,0.2)-- (-6.22,0.7);

\draw [line width=1.pt] (1.4,1.5)-- (0.62,0.2);
\draw [line width=1.pt,color=red] (0.62,0.2)-- (0.18,-0.7);
\draw [line width=1.pt] (1.4,1.5)-- (1.26,0.2);
\draw [line width=1.pt,color=red] (1.26,0.2)-- (1.16,-0.7);
\draw [line width=1.pt] (1.4,1.5)-- (2.7,0.2);
\draw [line width=1.pt,color=red] (2.7,0.2)-- (3.56,-0.7);
\draw (1.5,0.28) node[anchor=north west] {$\cdots$};
\begin{scriptsize}

\draw [fill=red] (-2.38,0.2) circle (2.5pt);
\draw [fill=red] (-2.72,-0.7) circle (2.5pt);
\draw[] (-2.72,-0.99) node {$1$-minimal tree};

\draw [fill=red] (-6.38,0.2) circle (2.5pt);
\draw [fill=red] (-6.72,-0.7) circle (2.5pt);
\draw[] (-6.72,-0.99) node {$\frac{3+\sqrt{5}}{2}$-L-minimal cut-tree};
\draw[] (-6.22,0.7) node {\Large\LeftScissors};

\draw [fill=red] (1.4,1.5) circle (2.5pt);
\draw [fill=red] (0.62,0.2) circle (2.5pt);
\draw [fill=red] (0.18,-0.7) circle (2.5pt);
\draw[color=red] (0.2,-1.2) node {$1$};
\draw [fill=red] (1.26,0.2) circle (2.5pt);
\draw [fill=red] (1.16,-0.7) circle (2.5pt);
\draw[color=red] (1.18,-1.2) node {$2$};
\draw [fill=red] (2.7,0.2) circle (2.5pt);
\draw [fill=red] (3.56,-0.7) circle (2.5pt);
\draw[color=red] (3.78,-1.2) node {$k+1$};
\end{scriptsize}
\end{tikzpicture}
\caption{The tree $T$ with minimum order and multiplicity $k$ ($k>1$) of $\lambda=1$ and $\mu=\frac{3+\sqrt{5}}{2}$.}
\end{figure}
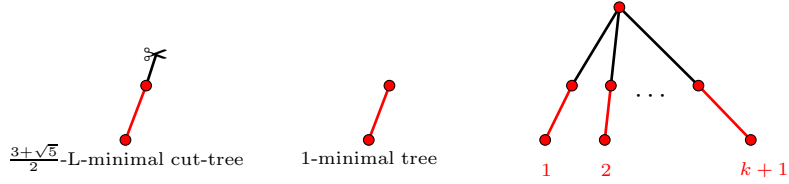

\end{example}

\begin{prob}
Suppose that $\lambda\in\mathtt{TRAI}$ and $\mu$ is a Laplacian eigenvalue of a tree. 
\begin{itemize}
	\item What is the size of  $\{T:\, T\in \mathcal{T}_{min,\lambda}\text{ with minimum order}\}?$
	\item If ${\rm Norm}(\mu)>1$, what is the size of $\{T:\, T\in\mathcal{LT}_{min,\mu}^{0}\text{ with minimum order}\}?$
\end{itemize}

\end{prob}
%theorem-------------------------------------------------------------------------------------
\vspace{.4cm}

\subsection{Totally Minimal Trees} 

	We say a tree $T$ is \textit{totally minimal} if for every $\lambda\in Spec(T)$, $T\in \mathcal{T}_{min,\lambda}$. 
By Theorem \ref{figen}, it is easy to see that $T$ is totally minimal if and only if for every $v\in V(T)$, $Spec(T)\cap Spec(T-v)=\emptyset$.

Suppose that $G$ is a graph of order $n$, $V(G)=\{u_1,\ldots,u_n\}$, and $H$ is a rooted graph with the root $v_1$.  The \textit{rooted product} $G\circ H$ is the graph that obtained from $G$ and $n$ copy $H_1,\ldots,H_n$ of $H$ by  identifying the root $v_1$ of $H_i$ with the vertex $u_i$ of $G$:
\[
V(G\circ H)=V(G)\times V(H),\, (u_i,v_j)\sim (u_k,v_l) \leftrightarrow (l=j=1\text{ and } u_i\sim u_k)\text{ or }(i=k\text{ and } v_j\sim l_l).
\]
\begin{figure}[H]
	\centering
	%\begin{tikzpicture}[scale=.9,line cap=round,line join=round,>=triangle 45,x=1.0cm,y=1.0cm]
	%\draw [line width=1.1pt] (1.22,2.74)-- (0.22,2.1);
	%\draw [line width=1.1pt] (1.22,2.74)-- (0.82,1.68);
	%\draw [line width=1.1pt,color=red] (0.82,1.68)-- (0.48,0.78);
	%\draw [line width=1.1pt] (1.22,2.74)-- (2.4,1.96);
	%\draw [line width=1.1pt,color=red] (2.4,1.96)-- (3.22,1.38);
	%\draw [line width=1.1pt,dotted,color=red] (3.22,1.38)-- (3.74,1.);
	%\draw [line width=1.1pt,color=red] (3.74,1.)-- (4.46,0.52);
	%\draw [color=red](-0.58,2.34) node[anchor=north west] {\Large $P_{p_{1}-1}$};
	%\draw [color=red](0.64,1.52) node[anchor=north west] {\Large $P_{p_{2}-1}$};
	%\draw [color=red](3.48,2.) node[anchor=north west] {\Large $P_{p_{k}-1}$};
	%\draw [color=red](1.12,2.16) node[anchor=north west] {\Large $\cdots$};
	%\begin{scriptsize}
	%\draw [fill=red] (1.22,2.74) circle (2.5pt);
	%\draw [fill=red] (0.22,2.1) circle (2.5pt);
	%\draw [fill=red] (0.82,1.68) circle (2.5pt);
	%\draw [fill=red] (0.48,0.78) circle (2.5pt);
	%\draw [fill=red] (2.4,1.96) circle (2.5pt);
	%\draw [fill=red] (3.22,1.38) circle (2.5pt);
	%\draw [fill=red] (3.74,1.) circle (2.5pt);
	%\draw [fill=red] (4.46,0.52) circle (2.5pt);
	%\end{scriptsize}
	%\end{tikzpicture}\\
	%\vspace{.5cm}
	
	\begin{tikzpicture}[scale=.7,line cap=round,line join=round,>=triangle 45,x=1.0cm,y=1.0cm]
	
	\draw [line width=1.1pt] (1.22,2.74)-- (0.22,2.1);
	\draw [line width=1.1pt] (1.22,2.74)-- (0.82,1.68);
	\draw [line width=1.1pt,color=red] (0.82,1.68)-- (0.48,0.78);
	%\draw [line width=1.1pt] (1.22,2.74)-- (2.4,1.96);
	%\draw [line width=1.1pt,color=red] (2.4,1.96)-- (3.22,1.38);
	%\draw [line width=1.1pt,dotted,color=red] (3.22,1.38)-- (3.74,1.);
	%\draw [line width=1.1pt,color=red] (3.74,1.)-- (4.46,0.52);
	%\draw [color=red](-0.78,2.45) node[anchor=north west] {$P_{p_{1}-1}$};
	%\draw [color=red](0.64,1.52) node[anchor=north west] {$P_{p_{2}-1}$};
	%\draw [color=red](3.08,2.) node[anchor=north west] {$P_{p_{k}-1}$};
	%\draw [color=red](1.02,2.16) node[anchor=north west] {\Large$\cdots$};
	\draw [line width=1.1pt] (5.74,2.74)-- (4.9,2.14);
	\draw [line width=1.1pt] (5.74,2.74)-- (5.34,1.67);
	\draw [line width=1.1pt,color=red] (5.34,1.67)-- (5.,0.77);
	%\draw [line width=1.1pt] (5.74,2.73)-- (6.92,1.95);
	%\draw [line width=1.1pt,color=red] (6.92,1.95)-- (7.74,1.37);
	%\draw [line width=1.1pt,dotted,color=red] (7.74,1.37)-- (8.26,0.99);
	%\draw [line width=1.1pt,color=red] (8.26,0.99)-- (8.98,0.51);
	%\draw [color=red](3.86,2.49) node[anchor=north west] {$P_{p_{1}-1}$};
	%\draw [color=red](5.16,1.5) node[anchor=north west] {$P_{p_{2}-1}$};
	%\draw [color=red](7.6,1.98) node[anchor=north west] {$P_{p_{k}-1}$};
	%\draw [color=red](5.54,2.14) node[anchor=north west] {\Large$\cdots$};
	\draw [line width=1.1pt] (1.22,2.74)-- (5.74,2.74);
	\draw [line width=1.1pt] (10.42,2.74)-- (9.58,2.14);
	\draw [line width=1.1pt] (10.42,2.74)-- (10.02,1.67);
	\draw [line width=1.1pt,color=red] (10.02,1.67)-- (9.68,0.77);
	
	\draw [line width=1.1pt] (-0.78,2.74)-- (-1.78,2.1);
    \draw [line width=1.1pt] (-0.78,2.74)-- (-1.18,1.68);
    \draw [line width=1.1pt,color=red] (-1.18,1.68)-- (-1.52,0.78);
    \draw [fill=red] (-0.78,2.74) circle (2.5pt);
    \draw[color=blue] (-0.61,3.12) node {\Large$v$};
    \draw [fill=red] (-1.78,2.1) circle (2.5pt);
    \draw [fill=red] (-1.18,1.68) circle (2.5pt);
    \draw [fill=red] (-1.52,0.78) circle (2.5pt);
	%\draw [line width=1.1pt] (10.42,2.73)-- (11.6,1.95);
	%\draw [line width=1.1pt,color=red] (11.6,1.95)-- (12.42,1.37);
	%\draw [line width=1.1pt,dotted,color=red] (12.42,1.37)-- (12.94,0.99);
	%\draw [line width=1.1pt,color=red] (12.94,0.99)-- (13.66,0.51);
	%\draw [color=red](9.84,1.5) node[anchor=north west] {$P_{p_{2}-1}$};
	%%\draw [color=red](12.48,1.98) node[anchor=north west] {$P_{p_{k}-1}$};
	%%\draw [color=red](10.22,2.14) node[anchor=north west] {\Large$\cdots$};
	%\draw [color=red](8.52,2.49) node[anchor=north west] {$P_{p_{1}-1}$};
	\draw [line width=1.1pt,dash pattern=on 2pt off 2pt] (5.74,2.74)-- (10.42,2.74);
	\begin{scriptsize}
	\draw [fill=red] (1.22,2.74) circle (2.5pt);
	\draw[color=blue] (1.29,3.12) node {\Large$v_1$};
	\draw [fill=red] (0.22,2.1) circle (2.5pt);
	\draw [fill=red] (0.82,1.68) circle (2.5pt);
	\draw [fill=red] (0.48,0.78) circle (2.5pt);

	%\draw [fill=red] (2.4,1.96) circle (2.5pt);
	%\draw [fill=red] (3.22,1.38) circle (2.5pt);
	%\draw [fill=red] (3.74,1.) circle (2.5pt);
	%\draw [fill=red] (4.46,0.52) circle (2.5pt);
	\draw [fill=red] (5.74,2.74) circle (2.5pt);
	\draw[color=blue] (6.05,3.14) node {\Large$v_2$};
	\draw [fill=red] (4.9,2.14) circle (2.5pt);
	\draw [fill=red] (5.34,1.67) circle (2.5pt);
	\draw [fill=red] (5.,0.77) circle (2.5pt);
	%\draw [fill=red] (6.92,1.95) circle (2.5pt);
	%\draw [fill=red] (7.74,1.37) circle (2.5pt);
	%\draw [fill=red] (8.26,0.99) circle (2.5pt);
	%\draw [fill=red] (8.98,0.51) circle (2.5pt);
	\draw [fill=red] (10.42,2.73) circle (2.5pt);
	\draw[color=blue] (10.55,3.08) node {\Large$v_{p-1}$};
	\draw [fill=red] (9.58,2.14) circle (2.5pt);
	\draw [fill=red] (10.02,1.67) circle (2.5pt);
	\draw [fill=red] (9.68,0.77) circle (2.5pt);
	%\draw [fill=red] (11.6,1.95) circle (2.5pt);
	%\draw [fill=red] (12.42,1.37) circle (2.5pt);
	%\draw [fill=red] (12.94,0.99) circle (2.5pt);
	%\draw [fill=red] (13.66,0.51) circle (2.5pt);
	\end{scriptsize}
	\end{tikzpicture}
	
	\caption{The trees $P_4$ and $P_{p-1}\circ P_4$.}
\end{figure}
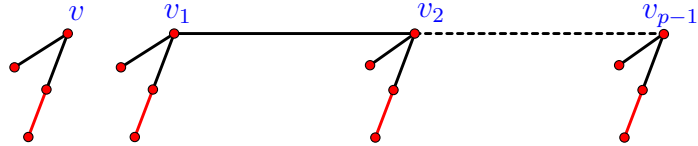
 As a corollary of \cite[Theorem 2.1]{GMc} we have the following theorem.

\begin{thm}{\cite[see Theorem 2.1]{GMc}}\label{gm}
Suppose that $G$ is a graph. If $H$ is a rooted graph, then 
\begin{equation}\label{gme}
\phi_{G\circ H}(\lambda)=|\phi_{H}(\lambda)I-\phi_{H-v}(\lambda)A_{G}|.
\end{equation}
\end{thm}
  
%For any $n_1,\ldots,n_k\in\mathbb{N}$, the starlike tree $S(n_{1},\ldots,n_{k})$ is a tree with a vertex of degree $k\geq 3$, say $v$, such that $S(n_{1},\ldots,n_{k})-v$ is the disjoint union of $P_{n_{1}},\ldots,P_{n_{k}}$.

\begin{thm}
Let $T$ be a totally minimal tree. Then the following trees are totally minimal:
\begin{enumerate}[i.]
	\item $P_{p-1}$, for any prime number $p$,
	\item $T\circ P_2$,
	\item $T\circ P_4$, where the root of $P_4$ is of degree $2$.
\end{enumerate}

\end{thm}

\begin{proof}
i. If $V(P_{p-1})=\{v_1,\ldots,v_{p-1}\}$, then ${\bf{\xi}}_{k}=(\sin(\frac{k\pi}{p}),\ldots,\sin(\frac{k(p-1)\pi}{p}))$ is a $\lambda_{k}$-eigenvector for $\lambda_{k}=2\cos(\frac{k\pi}{p})$, $k\in[p-1]$ \cite[Section 1.4]{BH}. Hence, we have $\sin(\frac{ki\pi}{p})\neq 0$, for $i,k\in[p-1]$.\\
ii., iii.: Suppose that $T$ is of order $n$, $V(T)=\{u_1,\ldots,u_n\}$, $S$ is a totally minimal tree, $V(S)=\{v_1,\ldots,v_m\}$, and  $v_1$ is the root of $S$. By Proposition \ref{mul1}, $\phi_{S}(\lambda)$ and $\phi_{S-v_1}(\lambda)$ do not have a common root, hence from equation (\ref{gme}), $\phi_{T\circ S}(\lambda)$ and $\phi_{S-v_1}(\lambda)$ do not have a common root. By Theorem \ref{gm} 
\[
\phi_{T\circ S}(\lambda)=(\phi_{S-v_1}(\lambda))^{n}|\frac{\phi_{S}(\lambda)}{\phi_{S-v_1}(\lambda)}I-A_{T}|=(\phi_{S-v_1}(\lambda))^{n} \phi_{T}(\frac{\phi_{S}(\lambda)}{\phi_{S-v_1}(\lambda)}).
\]
Therefore, the roots of $\phi_{T\circ S}(\lambda)$, are the roots of $\frac{\phi_{S}(\lambda)}{\phi_{S-v_1}(\lambda)}=\lambda_i$, where $\lambda_i$ is a root of $\phi_{T}(\lambda)$, for every $i\in [n]$.

It is sufficient to show that for any root $\theta$ of $\phi_{T\circ S}(\lambda)$ and any vertex $u=(u_i,v_j)$ of $T\circ S$, $\phi_{T\circ S -u}(\theta)\neq 0$. If $j=1$,
$$
\phi_{T\circ S -u}(\theta)=\phi_{(T-u_i)\circ S}(\theta) \cdot \phi_{S-v_1}(\theta)=(\phi_{S-v_1}(\theta))^{n} \phi_{T-u_i}(\frac{\phi_{S}(\theta)}{\phi_{S-v_1}(\theta)})\neq 0.
$$
Hence, any $\theta$-eigenvector does not vanish at $u$.

If $j>1$ and an eigenvalue vanishes at $u$, then by (\ref{eigeneqn}), it vanishes at $(u_i,v_1)$. Hence by the previous case we get a contradiction.
\end{proof}

%Construction--------------------------------------------------------------------------------------
We end this section with the following problem concerning totally minimal trees.
%\vspace{7mm} \noindent{\bf\Large II. Some Constructions} \vspace{5mm}
%\section{Some Constructions}
\begin{prob}
Which trees are totally minimal?	
\end{prob}

%\newpage

\section*{Appendix}

%\subsection*{Some tables of some $\mathcal{T}_{min,\lambda}$}

\begin{table}[H]
$\lambda=1:$\hspace*{3cm}$\lambda=\sqrt{2}:$\\
\begin{tabular}{|c|c|}
	\hline 
$n$	&  Tree  \\ 
	\hline 
$2$	&  \begin{tikzpicture}[scale=.5,line cap=round,line join=round,>=triangle 45,x=1.0cm,y=1.0cm]
\draw [line width=1.pt] (4.12,1.0)-- (7.02,1.0);
\begin{scriptsize}
\draw [fill=red] (4.12,1.0) circle (2.5pt);
\draw[color=blue] (4.26,1.45) node {$1$};
\draw [fill=red] (7.02,1.0) circle (2.5pt);
\draw[color=blue] (7.16,1.41) node {$1$};
\end{scriptsize}
\end{tikzpicture}  \\ 
	\hline 
 $6$	&  \begin{tikzpicture}[scale=.5,line cap=round,line join=round,>=triangle 45,x=1.0cm,y=1.0cm]
 \draw [line width=1.pt] (4.12,1.1)-- (6.98,1.1);
 \draw [line width=1.pt] (4.12,1.1)-- (2.68,2.16);
 \draw [line width=1.pt] (4.12,1.1)-- (2.76,0.2);
 \draw [line width=1.pt] (6.98,1.1)-- (8.08,2.24);
 \draw [line width=1.pt] (6.98,1.1)-- (8.26,0.14);
 \begin{scriptsize}
 \draw [fill=red] (4.12,1.1) circle (2.5pt);
 \draw[color=blue] (4.26,1.5) node {$1$};
 \draw [fill=red] (6.98,1.1) circle (2.5pt);
 \draw[color=blue] (6.78,1.5) node {$-1$};
 \draw [fill=red] (2.68,2.16) circle (2.5pt);
 \draw[color=blue] (2.82,2.53) node {$1$};
 \draw [fill=red] (2.76,0.2) circle (2.5pt);
 \draw[color=blue] (2.9,0.57) node {$1$};
 \draw [fill=red] (8.08,2.24) circle (2.5pt);
 \draw[color=blue] (8.28,2.61) node {$-1$};
 \draw [fill=red] (8.26,0.14) circle (2.5pt);
 \draw[color=blue] (8.46,0.5) node {$-1$};
 \end{scriptsize}
 \end{tikzpicture}  \\ 
	\hline 
$9$	&   \begin{tikzpicture}[scale=.5,line cap=round,line join=round,>=triangle 45,x=1.0cm,y=1.0cm]
\draw [line width=1.pt] (3.52,1.2)-- (2.68,2.16);
\draw [line width=1.pt] (3.52,1.2)-- (2.76,0.2);
\draw [line width=1.pt] (7.26,1.2)-- (8.08,2.2);
\draw [line width=1.pt] (7.26,1.2)-- (8.26,0.2);
\draw [line width=1.pt] (5.38,1.2)-- (3.52,1.2);
\draw [line width=1.pt] (5.38,1.2)-- (5.36,0.2);
\draw [line width=1.pt] (5.38,1.2)-- (7.26,1.2);
\draw [line width=1.pt] (7.26,1.2)-- (8.3,1.2);
\begin{scriptsize}
\draw [fill=red] (3.52,1.2) circle (2.5pt);
\draw[color=blue] (3.66,1.57) node {$1$};
\draw [fill=red] (7.26,1.2) circle (2.5pt);
\draw[color=blue] (7.12,1.59) node {$-1$};
\draw [fill=red] (2.68,2.2) circle (2.5pt);
\draw[color=blue] (2.82,2.53) node {$1$};
\draw [fill=red] (2.76,0.2) circle (2.5pt);
\draw[color=blue] (2.64,0.69) node {$-1$};
\draw [fill=red] (8.08,2.2) circle (2.5pt);
\draw[color=blue] (8.28,2.61) node {$-1$};
\draw [fill=red] (8.26,0.2) circle (2.5pt);
\draw[color=blue] (8.46,0.51) node {$-1$};
\draw [fill=red] (5.38,1.2) circle (2.5pt);
\draw[color=blue] (5.52,1.51) node {$2$};
\draw [fill=red] (5.36,0.2) circle (2.5pt);
\draw[color=blue] (4.98,0.21) node {$2$};
\draw [fill=red] (8.3,1.2) circle (2.5pt);
\draw[color=blue] (8.5,1.65) node {$-1$};
\end{scriptsize}
\end{tikzpicture} \\ 
\hline 	
$9$	&   \begin{tikzpicture}[scale=.5,line cap=round,line join=round,>=triangle 45,x=1.0cm,y=1.0cm]
\draw [line width=1.pt] (3.52,1.2)-- (2.68,2.16);
\draw [line width=1.pt] (3.52,1.2)-- (2.76,0.2);
\draw [line width=1.pt] (7.26,1.2)-- (8.08,2.24);
\draw [line width=1.pt] (7.26,1.2)-- (8.26,0.14);
\draw [line width=1.pt] (3.52,1.2)-- (4.34,1.2);
\draw [line width=1.pt] (4.34,1.2)-- (5.34,1.2);
\draw [line width=1.pt] (5.34,1.2)-- (6.3,1.2);
\draw [line width=1.pt] (6.3,1.2)-- (7.26,1.2);
\begin{scriptsize}
\draw [fill=red] (3.52,1.2) circle (2.5pt);
\draw[color=blue] (3.72,1.57) node {$-1$};
\draw [fill=red] (7.26,1.2) circle (2.5pt);
\draw[color=blue] (7.06,1.55) node {$-1$};
\draw [fill=red] (2.68,2.16) circle (2.5pt);
\draw[color=blue] (2.88,2.53) node {$-1$};
\draw [fill=red] (2.76,0.2) circle (2.5pt);
\draw[color=blue] (2.6,0.63) node {$-1$};
\draw [fill=red] (8.08,2.24) circle (2.5pt);
\draw[color=blue] (8.28,2.61) node {$-1$};
\draw [fill=red] (8.26,0.14) circle (2.5pt);
\draw[color=blue] (8.46,0.51) node {$-1$};
\draw [fill=red] (4.34,1.2) circle (2.5pt);
\draw[color=blue] (4.48,1.55) node {$1$};
\draw [fill=red] (5.34,1.2) circle (2.5pt);
\draw[color=blue] (5.48,1.51) node {$2$};
\draw [fill=red] (6.3,1.2) circle (2.5pt);
\draw[color=blue] (6.44,1.49) node {$1$};
\end{scriptsize}
\end{tikzpicture} \\ 
	\hline 
$10$	& \begin{tikzpicture}[scale=.5,line cap=round,line join=round,>=triangle 45,x=1.0cm,y=1.0cm]
\draw [line width=1.pt] (3.52,1.2)-- (2.68,2.2);
\draw [line width=1.pt] (3.52,1.2)-- (2.76,0.2);
\draw [line width=1.pt] (7.26,1.2)-- (8.08,2.2);
\draw [line width=1.pt] (7.26,1.2)-- (8.26,0.2);
\draw [line width=1.pt] (5.4,1.2)-- (3.52,1.2);
\draw [line width=1.pt] (5.4,1.2)-- (7.26,1.2);
\draw [line width=1.pt] (5.4,1.2)-- (4.48,0.2);
\draw [line width=1.pt] (5.4,1.2)-- (6.28,0.2);
\draw [line width=1.pt] (5.4,1.2)-- (5.34,0.2);
\begin{scriptsize}
\draw [fill=red] (3.52,1.2) circle (2.5pt);
\draw[color=blue] (3.72,1.57) node {$-1$};
\draw [fill=red] (7.26,1.2) circle (2.5pt);
\draw[color=blue] (7.16,1.55) node {$-1$};
\draw [fill=red] (2.68,2.2) circle (2.5pt);
\draw[color=blue] (2.88,2.53) node {$-1$};
\draw [fill=red] (2.76,0.2) circle (2.5pt);
\draw[color=blue] (2.6,0.63) node {$-1$};
\draw [fill=red] (8.08,2.2) circle (2.5pt);
\draw[color=blue] (8.28,2.61) node {$-1$};
\draw [fill=red] (8.26,0.2) circle (2.5pt);
\draw[color=blue] (8.46,0.51) node {$-1$};
\draw [fill=red] (5.4,1.2) circle (2.5pt);
\draw[color=blue] (5.54,1.55) node {$1$};
\draw [fill=red] (4.48,0.2) circle (2.5pt);
\draw[color=blue] (4.2,0.29) node {$1$};
\draw [fill=red] (6.28,0.2) circle (2.5pt);
\draw[color=blue] (6.58,0.35) node {$1$};
\draw [fill=red] (5.34,0.2) circle (2.5pt);
\draw[color=blue] (5.64,0.2) node {$1$};
\end{scriptsize}
\end{tikzpicture}   \\ 
	\hline 
$10$	&  \begin{tikzpicture}[scale=.5,line cap=round,line join=round,>=triangle 45,x=1.0cm,y=1.0cm]
\draw [line width=1.pt] (3.52,1.2)-- (2.68,2.2);
\draw [line width=1.pt] (3.52,1.2)-- (2.76,0.2);
\draw [line width=1.pt] (7.26,1.2)-- (8.08,2.2);
\draw [line width=1.pt] (7.26,1.2)-- (8.26,0.2);
\draw [line width=1.pt] (4.86,1.2)-- (3.52,1.2);
\draw [line width=1.pt] (4.86,1.2)-- (4.78,0.2);
\draw [line width=1.pt] (4.86,1.2)-- (5.98,1.2);
\draw [line width=1.pt] (5.98,1.2)-- (7.26,1.2);
\draw [line width=1.pt] (5.98,1.2)-- (5.98,0.2);
\begin{scriptsize}
\draw [fill=red] (3.52,1.2) circle (2.5pt);
\draw[color=blue] (3.72,1.57) node {$-1$};
\draw [fill=red] (7.26,1.2) circle (2.5pt);
\draw[color=blue] (7.12,1.59) node {$-1$};
\draw [fill=red] (2.68,2.2) circle (2.5pt);
\draw[color=blue] (2.88,2.53) node {$-1$};
\draw [fill=red] (2.76,0.2) circle (2.5pt);
\draw[color=blue] (2.64,0.69) node {$-1$};
\draw [fill=red] (8.08,2.2) circle (2.5pt);
\draw[color=blue] (8.28,2.61) node {$-1$};
\draw [fill=red] (8.26,0.2) circle (2.5pt);
\draw[color=blue] (8.46,0.51) node {$-1$};
\draw [fill=red] (4.86,1.2) circle (2.5pt);
\draw[color=blue] (5.,1.55) node {$1$};
\draw [fill=red] (4.78,0.2) circle (2.5pt);
\draw[color=blue] (4.4,0.27) node {$1$};
\draw [fill=red] (5.98,1.2) circle (2.5pt);
\draw[color=blue] (6.12,1.55) node {$1$};
\draw [fill=red] (5.98,0.2) circle (2.5pt);
\draw[color=blue] (6.34,0.25) node {$1$};
\end{scriptsize}
\end{tikzpicture} \\ 
	\hline 

\end{tabular} 
\begin{tabular}{|c|c|}
		\hline 
		$n$	&  Tree  \\ 
		\hline 
		$3$	&  \begin{tikzpicture}[scale=.5,line cap=round,line join=round,>=triangle 45,x=1.0cm,y=1.0cm]
		\draw [line width=1.pt] (2.48,1.2)-- (4.42,1.2);
		\draw [line width=1.pt] (4.42,1.2)-- (6.34,1.2);
		\begin{scriptsize}
		\draw [fill=red] (2.48,1.2) circle (2.5pt);
		\draw[color=blue] (2.74,1.55) node {$1$};
		\draw [fill=red] (4.42,1.2) circle (2.5pt);
		\draw[color=blue] (5.04,1.53) node {$\sqrt{2}$};
		\draw [fill=red] (6.34,1.2) circle (2.5pt);
		\draw[color=blue] (6.6,1.53) node {$1$};
		\end{scriptsize}
		\end{tikzpicture}  \\ 
		\hline 
$9$ & \begin{tikzpicture}[scale=.5,line cap=round,line join=round,>=triangle 45,x=1.0cm,y=1.0cm]
\draw [line width=1.pt] (2.94,1.2)-- (4.42,1.2);
\draw [line width=1.pt] (4.42,1.2)-- (6.34,1.2);
\draw [line width=1.pt] (6.34,1.2)-- (7.24,2.2);
\draw [line width=1.pt] (6.34,1.2)-- (5.98,0.2);
\draw [line width=1.pt] (6.34,1.2)-- (6.04,2.2);
\draw [line width=1.pt] (4.42,1.2)-- (3.86,2.2);
\draw [line width=1.pt] (4.42,1.2)-- (3.82,0.2);
\draw [line width=1.pt] (6.34,1.2)-- (7.28,0.2);
\begin{scriptsize}
\draw [fill=red] (2.94,1.2) circle (2.5pt);
\draw[color=blue] (2.22,1.17) node {$-\sqrt{2}$};
\draw [fill=red] (4.42,1.2) circle (2.5pt);
\draw[color=blue] (4.67,1.39) node {$-2$};
\draw [fill=red] (6.34,1.2) circle (2.5pt);
\draw[color=blue] (5.8,1.55) node {$\sqrt{2}$};
\draw [fill=red] (7.24,2.2) circle (2.5pt);
\draw[color=blue] (7.34,1.61) node {$1$};
\draw [fill=red] (5.98,0.2) circle (2.5pt);
\draw[color=blue] (5.72,0.21) node {$1$};
\draw [fill=red] (7.28,0.2) circle (2.5pt);
\draw[color=blue] (7.3,0.81) node {$1$};
\draw [fill=red] (6.04,2.2) circle (2.5pt);
\draw[color=blue] (5.72,2.17) node {$1$};
\draw [fill=red] (3.86,2.2) circle (2.5pt);
\draw[color=blue] (3.18,2.25) node {$-\sqrt{2}$};
\draw [fill=red] (3.82,0.2) circle (2.5pt);
\draw[color=blue] (3.14,0.29) node {$-\sqrt{2}$};
\end{scriptsize}
\end{tikzpicture} \\
				\hline 
$9$	& \begin{tikzpicture}[scale=.5,line cap=round,line join=round,>=triangle 45,x=1.0cm,y=1.0cm]
\draw [line width=1.pt] (2.94,1.2)-- (4.42,1.2);
\draw [line width=1.pt] (4.42,1.2)-- (6.34,1.2);
\draw [line width=1.pt] (6.34,1.2)-- (7.24,2.2);
\draw [line width=1.pt] (6.34,1.2)-- (5.98,0.2);
\draw [line width=1.pt] (6.34,1.2)-- (6.04,2.2);
\draw [line width=1.pt] (4.42,1.2)-- (3.82,0.2);
\draw [line width=1.pt] (6.34,1.2)-- (7.28,0.2);
\draw [line width=1.pt] (2.94,1.2)-- (1.8,1.2);
\begin{scriptsize}
\draw [fill=red] (2.94,1.2) circle (2.5pt);
\draw[color=blue] (3.16,1.61) node {$-2\sqrt{2}$};
\draw [fill=red] (4.42,1.2) circle (2.5pt);
\draw[color=blue] (4.46,1.49) node {$-2$};
\draw [fill=red] (6.34,1.2) circle (2.5pt);
\draw[color=blue] (5.8,1.5) node {$\sqrt{2}$};
\draw [fill=red] (7.24,2.2) circle (2.5pt);
\draw[color=blue] (7.34,1.57) node {$1$};
\draw [fill=red] (5.98,0.2) circle (2.5pt);
\draw[color=blue] (5.72,0.17) node {$1$};
\draw [fill=red] (7.28,0.2) circle (2.5pt);
\draw[color=blue] (7.3,0.77) node {$1$};
\draw [fill=red] (6.04,2.2) circle (2.5pt);
\draw[color=blue] (5.72,2.13) node {$1$};
\draw [fill=red] (3.82,0.2) circle (2.5pt);
\draw[color=blue] (3.14,0.25) node {$-\sqrt{2}$};
\draw [fill=red] (1.8,1.2) circle (2.5pt);
\draw[color=blue] (1.7,1.51) node {$-2$};
\end{scriptsize}
\end{tikzpicture}\\
	\hline			
$9$	& \begin{tikzpicture}[scale=.5,line cap=round,line join=round,>=triangle 45,x=1.0cm,y=1.0cm]
\draw [line width=1.pt] (2.94,1.2)-- (4.42,1.2);
\draw [line width=1.pt] (4.42,1.2)-- (6.34,1.2);
\draw [line width=1.pt] (6.34,1.2)-- (7.24,2.2);
\draw [line width=1.pt] (6.34,1.2)-- (5.98,0.2);
\draw [line width=1.pt] (4.42,1.2)-- (3.82,0.2);
\draw [line width=1.pt] (6.34,1.2)-- (7.28,0.2);
\draw [line width=1.pt] (2.94,1.2)-- (1.8,1.2);
\draw [line width=1.pt] (4.42,1.2)-- (3.9,2.2);
\begin{scriptsize}
\draw [fill=red] (2.94,1.2) circle (2.5pt);
\draw[color=blue] (2.8,1.57) node {$-2$};
\draw [fill=red] (4.42,1.2) circle (2.5pt);
\draw[color=blue] (4.94,1.57) node {$-\sqrt{2}$};
\draw [fill=red] (6.34,1.2) circle (2.5pt);
\draw[color=blue] (6.1,1.49) node {$2$};
\draw [fill=red] (7.24,2.2) circle (2.5pt);
\draw[color=blue] (7.7,1.57) node {$\sqrt{2}$};
\draw [fill=red] (5.98,0.2) circle (2.5pt);
\draw[color=blue] (5.66,0.57) node {$\sqrt{2}$};
\draw [fill=red] (7.28,0.2) circle (2.5pt);
\draw[color=blue] (7.66,0.77) node {$\sqrt{2}$};
\draw [fill=red] (3.82,0.2) circle (2.5pt);
\draw[color=blue] (3.34,0.33) node {$-1$};
\draw [fill=red] (1.8,1.2) circle (2.5pt);
\draw[color=blue] (1.7,1.56) node {$-\sqrt{2}$};
\draw [fill=red] (3.9,2.2) circle (2.5pt);
\draw[color=blue] (4.3,2.15) node {$-1$};
\end{scriptsize}
\end{tikzpicture}\\
\hline	
$9$	& \begin{tikzpicture}[scale=.5,line cap=round,line join=round,>=triangle 45,x=1.0cm,y=1.0cm]
\draw [line width=1.pt] (2.94,1.2)-- (4.42,1.2);
\draw [line width=1.pt] (4.42,1.2)-- (6.24,1.2);
\draw [line width=1.pt] (6.24,1.2)-- (7.24,2.2);
\draw [line width=1.pt] (6.24,1.2)-- (5.98,0.2);
\draw [line width=1.pt] (4.42,1.2)-- (3.82,0.2);
\draw [line width=1.pt] (6.24,1.2)-- (7.28,0.2);
\draw [line width=1.pt] (2.94,1.2)-- (1.8,1.2);
\draw [line width=1.pt] (3.82,0.2)-- (2.8,0.2);
\begin{scriptsize}
\draw [fill=red] (2.94,1.2) circle (2.5pt);
\draw[color=blue] (3.16,1.57) node {$-\sqrt{2}$};
\draw [fill=red] (4.42,1.2) circle (2.5pt);
\draw[color=blue] (4.58,1.57) node {$-1$};
\draw [fill=red] (6.24,1.2) circle (2.5pt);
\draw[color=blue] (6.0,1.63) node {$\sqrt{2}$};
\draw [fill=red] (7.24,2.2) circle (2.5pt);
\draw[color=blue] (7.34,1.57) node {$1$};
\draw [fill=red] (5.98,0.2) circle (2.5pt);
\draw[color=blue] (5.66,0.41) node {$1$};
\draw [fill=red] (7.28,0.2) circle (2.5pt);
\draw[color=blue] (7.3,0.77) node {$1$};
\draw [fill=red] (3.82,0.2) circle (2.5pt);
\draw[color=blue] (3.18,0.55) node {$-\sqrt{2}$};
\draw [fill=red] (1.8,1.2) circle (2.5pt);
\draw[color=blue] (1.7,1.51) node {$-1$};
\draw [fill=red] (2.8,0.2) circle (2.5pt);
\draw[color=blue] (2.28,0.27) node {$-1$};
\end{scriptsize}
\end{tikzpicture}\\
\hline	
$9$	& \begin{tikzpicture}[scale=.5,line cap=round,line join=round,>=triangle 45,x=1.0cm,y=1.0cm]
\draw [line width=1.pt] (2.94,1.2)-- (4.42,1.2);
\draw [line width=1.pt] (4.42,1.2)-- (6.24,1.2);
\draw [line width=1.pt] (6.24,1.2)-- (7.24,2.2);
\draw [line width=1.pt] (6.24,1.2)-- (5.98,0.2);
\draw [line width=1.pt] (4.42,1.2)-- (3.82,0.2);
\draw [line width=1.pt] (6.24,1.2)-- (7.28,0.2);
\draw [line width=1.pt] (2.94,1.2)-- (1.8,1.2);
\draw [line width=1.pt] (7.28,0.2)-- (8.66,0.2);
\begin{scriptsize}
\draw [fill=red] (2.94,1.2) circle (2.5pt);
\draw[color=blue] (3.22,1.57) node {$-2\sqrt{2}$};
\draw [fill=red] (4.42,1.2) circle (2.5pt);
\draw[color=blue] (4.48,1.57) node {$-2$};
\draw [fill=red] (6.24,1.2) circle (2.5pt);
\draw[color=blue] (6.0,1.63) node {$\sqrt{2}$};
\draw [fill=red] (7.24,2.2) circle (2.5pt);
\draw[color=blue] (7.34,1.57) node {$1$};
\draw [fill=red] (5.98,0.2) circle (2.5pt);
\draw[color=blue] (5.66,0.41) node {$1$};
\draw [fill=red] (7.28,0.2) circle (2.5pt);
\draw[color=blue] (7.36,0.69) node {$2$};
\draw [fill=red] (3.82,0.2) circle (2.5pt);
\draw[color=blue] (3.3,0.55) node {$-\sqrt{2}$};
\draw [fill=red] (1.8,1.2) circle (2.5pt);
\draw[color=blue] (1.7,1.51) node {$-2$};
\draw [fill=red] (8.66,0.2) circle (2.5pt);
\draw[color=blue] (9.28,0.67) node {$\sqrt{2}$};
\end{scriptsize}
\end{tikzpicture} \\
\hline					
$9$	& \begin{tikzpicture}[scale=.5,line cap=round,line join=round,>=triangle 45,x=1.0cm,y=1.0cm]
\draw [line width=1.pt] (2.94,1.2)-- (4.42,1.2);
\draw [line width=1.pt] (4.42,1.2)-- (6.24,1.2);
\draw [line width=1.pt] (6.24,1.2)-- (7.24,2.2);
\draw [line width=1.pt] (4.42,1.2)-- (3.82,0.2);
\draw [line width=1.pt] (6.24,1.2)-- (7.28,0.2);
\draw [line width=1.pt] (2.94,1.2)-- (1.8,1.2);
\draw [line width=1.pt] (7.28,0.2)-- (8.66,0.2);
\draw [line width=1.pt] (7.24,2.2)-- (8.7,2.2);
\begin{scriptsize}
\draw [fill=red] (2.94,1.2) circle (2.5pt);
\draw[color=blue] (2.8,1.57) node {$-2$};
\draw [fill=red] (4.42,1.2) circle (2.5pt);
\draw[color=blue] (4.84,1.57) node {$-\sqrt{2}$};
\draw [fill=red] (6.24,1.2) circle (2.5pt);
\draw[color=blue] (6.04,1.49) node {$1$};
\draw [fill=red] (7.24,2.2) circle (2.5pt);
\draw[color=blue] (6.5,2.05) node {$\sqrt{2}$};
\draw [fill=red] (7.28,0.2) circle (2.5pt);
\draw[color=blue] (6.5,0.39) node {$\sqrt{2}$};
\draw [fill=red] (3.82,0.2) circle (2.5pt);
\draw[color=blue] (3.32,0.29) node {$-1$};
\draw [fill=red] (1.8,1.2) circle (2.5pt);
\draw[color=blue] (1.7,1.51) node {$-\sqrt{2}$};
\draw [fill=red] (8.66,0.2) circle (2.5pt);
\draw[color=blue] (9.16,0.33) node {$1$};
\draw [fill=red] (8.7,2.2) circle (2.5pt);
\draw[color=blue] (9.2,2.01) node {$1$};
\end{scriptsize}
\end{tikzpicture}\\
\hline					
	\end{tabular} 
	\caption{Minimal trees of $\lambda=1$ and $\lambda=\sqrt{2}$ with at most $10$ vertices.}
\end{table}

%\subsection*{Some tables of some $\mathcal{LT}_{min,\mu}$}

\begin{table}[H]
\begin{tabular}{|c|c|c|c|}
	\hline 
Minimal Polynomial	& $\mu$  & $\mu$-L-minimal $0$-cut-tree & $\mu$-L-minimal $1$-cut-tree \\ 
	\hline 
	&&&\\
& & \begin{tikzpicture}[scale=.5,line cap=round,line join=round,>=triangle 45,x=1.0cm,y=1.0cm]
\draw [line width=1.pt] (5.14,1)-- (2.58,1);
\draw [line width=1.pt] (2.58,1)-- (2.6,0);
\draw [line width=1.pt] (5.14,1)-- (7.28,1);
\draw [line width=1.pt] (7.28,1)-- (5.66,0);
\draw [line width=1.pt] (7.28,1)-- (6.48,0);
\draw [line width=1.pt] (7.28,1)-- (7.18,0);
\draw [line width=1.pt] (7.28,1)-- (8.08,0);
\draw [line width=1.pt] (7.28,1)-- (8.9,0);
\draw [line width=1.pt] (8.9,0)-- (9.36,-1);
\draw [line width=1.pt] (8.9,0)-- (8.08,-1);
\draw [line width=1.pt] (2.6,0)-- (3.48,-1);
\draw [line width=1.pt] (2.6,0)-- (2.,-1);
\begin{scriptsize}
\draw [fill=red] (5.14,1) circle (2.5pt);
\draw [fill=red] (2.58,1) circle (2.5pt);
\draw [fill=red] (2.6,0) circle (2.5pt);
\draw [fill=red] (7.28,1) circle (2.5pt);
\draw [fill=red] (5.66,0) circle (2.5pt);
\draw [fill=red] (6.48,0) circle (2.5pt);
\draw [fill=red] (7.18,0) circle (2.5pt);
\draw [fill=red] (8.08,0) circle (2.5pt);
\draw [fill=red] (8.9,0) circle (2.5pt);
\draw [fill=red] (9.36,-1) circle (2.5pt);
\draw [fill=red] (8.08,-1) circle (2.5pt);
\draw [fill=red] (3.48,-1) circle (2.5pt);
\draw [fill=red] (2.,-1) circle (2.5pt);
\end{scriptsize}
\end{tikzpicture} & \begin{tikzpicture}[scale=.5,line cap=round,line join=round,>=triangle 45,x=1.0cm,y=1.0cm]

\draw [line width=1.pt] (1.82,0.2)-- (3.58,0.2);
\draw [line width=1.pt] (3.58,0.2)-- (4.5,0.2);
\begin{scriptsize}
\draw [fill=red] (1.82,0.2) circle (2.5pt);
\draw [fill=red] (3.58,0.2) circle (2.5pt);
\draw[] (4.5,0.2) node {\Large\RightScissors};
\end{scriptsize}
\end{tikzpicture} \\ 
$x^2 -3x+1$	& $\frac{3\pm \sqrt{5}}{2}$& \begin{tikzpicture}[scale=.5,line cap=round,line join=round,>=triangle 45,x=1.0cm,y=1.0cm]
\draw [line width=1.pt] (4.46,1.5)-- (2.58,1.5);
\draw [line width=1.pt] (2.58,1.5)-- (2.6,-0.0);
\draw [line width=1.pt] (8.94,1.5)-- (8.9,0.);
\draw [line width=1.pt] (8.9,0.)-- (9.36,-1.5);
\draw [line width=1.pt] (8.9,0.)-- (8.08,-1.5);
\draw [line width=1.pt] (2.6,-0.0)-- (3.48,-1.5);
\draw [line width=1.pt] (2.6,-0.0)-- (2.,-1.5);
\draw [line width=1.pt] (2.58,1.5)-- (3.72,0.0);
\draw [line width=1.pt] (8.94,1.5)-- (7.78,0.);
\draw [line width=1.pt] (4.46,1.5)-- (6.14,1.5);
\draw [line width=1.pt] (6.14,1.5)-- (7.36,1.5);
\draw [line width=1.pt] (7.36,1.5)-- (8.94,1.5);
\begin{scriptsize}
\draw [fill=red] (4.46,1.5) circle (2.5pt);
\draw [fill=red] (2.58,1.5) circle (2.5pt);
\draw [fill=red] (2.6,-0.0) circle (2.5pt);
\draw [fill=red] (8.94,1.5) circle (2.5pt);
\draw [fill=red] (8.9,0.) circle (2.5pt);
\draw [fill=red] (9.36,-1.5) circle (2.5pt);
\draw [fill=red] (8.08,-1.5) circle (2.5pt);
\draw [fill=red] (3.48,-1.5) circle (2.5pt);
\draw [fill=red] (2.,-1.5) circle (2.5pt);
\draw [fill=red] (3.72,0.0) circle (2.5pt);
\draw [fill=red] (7.78,0.) circle (2.5pt);
\draw [fill=red] (6.14,1.5) circle (2.5pt);
\draw [fill=red] (7.36,1.5) circle (2.5pt);
\end{scriptsize}
\end{tikzpicture}  & \begin{tikzpicture}[scale=.5,line cap=round,line join=round,>=triangle 45,x=1.0cm,y=1.0cm]
\draw [line width=1.pt] (3.52,1.2)-- (2.68,2.16);
\draw [line width=1.pt] (3.52,1.2)-- (2.76,0.2);
\draw [line width=1.pt] (7.26,1.2)-- (7.26,2.2);
\draw [line width=1.pt] (7.26,1.2)-- (8.26,1.2);
\draw [line width=1.pt] (5.38,1.2)-- (3.52,1.2);
\draw [line width=1.pt] (5.38,1.2)-- (5.36,0.2);
\draw [line width=1.pt] (5.38,1.2)-- (7.26,1.2);
\draw [line width=1.pt] (7.26,1.2)-- (7.26,0.2);
\begin{scriptsize}
\draw [fill=red] (3.52,1.2) circle (2.5pt);
\draw [fill=red] (7.26,1.2) circle (2.5pt);
\draw [fill=red] (2.68,2.2) circle (2.5pt);
\draw [fill=red] (2.76,0.2) circle (2.5pt);
\draw [fill=red] (7.26,2.2) circle (2.5pt);
\draw [fill=red] (7.26,0.2) circle (2.5pt);
\draw [fill=red] (5.38,1.2) circle (2.5pt);
\draw [fill=red] (5.36,0.2) circle (2.5pt);
%\draw [fill=red] (8.3,1.2) circle (2.5pt);
\draw[] (8.3,1.2) node {\Large\RightScissors};
\end{scriptsize}
\end{tikzpicture} \\
&& \begin{tikzpicture}[scale=.5,line cap=round,line join=round,>=triangle 45,x=1.0cm,y=1.0cm]
\draw [line width=1.pt] (2.6,1.5)-- (2.6,-0.);
\draw [line width=1.pt] (8.94,1.5)-- (8.9,0.);
\draw [line width=1.pt] (8.9,0.)-- (9.36,-1.5);
\draw [line width=1.pt] (8.9,0.)-- (8.08,-1.5);
\draw [line width=1.pt] (2.6,-0.0)-- (3.48,-1.5);
\draw [line width=1.pt] (2.6,-0.0)-- (2.,-1.5);
\draw [line width=1.pt] (2.6,1.5)-- (3.72,0.0);
\draw [line width=1.pt] (8.94,1.5)-- (7.78,0.);
\draw [line width=1.pt] (2.58,1.5)-- (6.,1.5);
\draw [line width=1.pt] (6.,1.5)-- (8.94,1.5);
\draw [line width=1.pt] (6.,1.5)-- (6,0.0);
\draw [line width=1.pt] (2.6,1.5)-- (1.68,0.);
\draw [line width=1.pt] (8.94,1.5)-- (9.98,0.);
\begin{scriptsize}
\draw [fill=red] (2.6,1.5) circle (2.5pt);
\draw [fill=red] (2.6,-0.0) circle (2.5pt);
\draw [fill=red] (8.94,1.5) circle (2.5pt);
\draw [fill=red] (8.9,0.) circle (2.5pt);
\draw [fill=red] (9.36,-1.5) circle (2.5pt);
\draw [fill=red] (8.08,-1.5) circle (2.5pt);
\draw [fill=red] (3.48,-1.5) circle (2.5pt);
\draw [fill=red] (2.,-1.5) circle (2.5pt);
\draw [fill=red] (3.72,0.0) circle (2.5pt);
\draw [fill=red] (7.78,0.) circle (2.5pt);
\draw [fill=red] (6.,1.5) circle (2.5pt);
\draw [fill=red] (6,0.0) circle (2.5pt);
\draw [fill=red] (1.68,0.) circle (2.5pt);
\draw [fill=red] (9.98,0.) circle (2.5pt);
\end{scriptsize}
\end{tikzpicture} & \begin{tikzpicture}[scale=.5,line cap=round,line join=round,>=triangle 45,x=1.0cm,y=1.0cm]
\draw [line width=1.pt] (3.52,1.2)-- (2.68,2.16);
\draw [line width=1.pt] (3.52,1.2)-- (2.76,0.2);
\draw [line width=1.pt] (7.26,1.2)-- (8.08,2.2);
\draw [line width=1.pt] (7.26,1.2)-- (8.26,0.2);
\draw [line width=1.pt] (5.38,1.2)-- (3.52,1.2);
\draw [line width=1.pt] (5.38,1.2)-- (5.36,0.2);
\draw [line width=1.pt] (5.38,1.2)-- (7.26,1.2);
\draw [line width=1.pt] (8.08,2.2)-- (9.0,2.2);
\begin{scriptsize}
\draw [fill=red] (3.52,1.2) circle (2.5pt);
\draw [fill=red] (7.26,1.2) circle (2.5pt);
\draw [fill=red] (2.68,2.2) circle (2.5pt);
\draw [fill=red] (2.76,0.2) circle (2.5pt);
\draw [fill=red] (8.08,2.2) circle (2.5pt);
\draw [fill=red] (8.26,0.2) circle (2.5pt);
\draw [fill=red] (5.38,1.2) circle (2.5pt);
\draw [fill=red] (5.36,0.2) circle (2.5pt);
%\draw [fill=red] (8.3,1.2) circle (2.5pt);
\draw[] (9.0,2.2) node {\Large\RightScissors};
\end{scriptsize}
\end{tikzpicture} \\
	\hline 
		&&&\\
& & \begin{tikzpicture}[scale=.5,line cap=round,line join=round,>=triangle 45,x=1.0cm,y=1.0cm]
\draw [line width=1.pt] (1.18,-0.4)-- (2.56,-0.4);
\draw [line width=1.pt] (2.56,-0.4)-- (3.92,-0.4);
\draw [line width=1.pt] (3.92,-0.4)-- (5.34,-0.4);
\draw [line width=1.pt] (5.34,-0.4)-- (6.8,-0.4);
\draw [line width=1.pt] (6.8,-0.4)-- (7.98,-0.4);
\begin{scriptsize}
\draw [fill=red] (1.18,-0.4) circle (2.5pt);
\draw [fill=red] (2.56,-0.4) circle (2.5pt);
\draw [fill=red] (3.92,-0.4) circle (2.5pt);
\draw [fill=red] (5.34,-0.4) circle (2.5pt);
\draw [fill=red] (6.8,-0.4) circle (2.5pt);
\draw [fill=red] (7.98,-0.4) circle (2.5pt);
\end{scriptsize}
\end{tikzpicture} & \begin{tikzpicture}[scale=.5,line cap=round,line join=round,>=triangle 45,x=1.0cm,y=1.0cm]
\draw [line width=1.pt] (2.06,0.2)-- (3.58,0.2);
\draw [line width=1.pt] (2.06,0.2)-- (0.58,0.2);
\draw [line width=1.pt] (2.06,0.2)-- (2.06,1.0);
\begin{scriptsize}
\draw [fill=red] (2.06,0.2) circle (2.5pt);
\draw [fill=red] (3.58,0.2) circle (2.5pt);
\draw [fill=red] (0.58,0.2) circle (2.5pt);
\draw[] (2.06,1.0) node {\Large\RightScissors};
\end{scriptsize}
\end{tikzpicture} \\ 
$x^2 -4x+1$	& $2\pm \sqrt{3}$& \begin{tikzpicture}[scale=.5,line cap=round,line join=round,>=triangle 45,x=1.0cm,y=1.0cm]
\draw [line width=1.pt] (2.94,1.2)-- (4.42,1.2);
\draw [line width=1.pt] (4.42,1.2)-- (6.34,1.2);
\draw [line width=1.pt] (6.34,1.2)-- (7.24,2.2);
\draw [line width=1.pt] (6.34,1.2)-- (5.98,0.2);
\draw [line width=1.pt] (6.34,1.2)-- (6.04,2.2);
\draw [line width=1.pt] (4.42,1.2)-- (3.82,0.2);
\draw [line width=1.pt] (6.34,1.2)-- (7.28,0.2);
\draw [line width=1.pt] (2.94,1.2)-- (1.8,1.2);
\begin{scriptsize}
\draw [fill=red] (2.94,1.2) circle (2.5pt);

\draw [fill=red] (4.42,1.2) circle (2.5pt);

\draw [fill=red] (6.34,1.2) circle (2.5pt);

\draw [fill=red] (7.24,2.2) circle (2.5pt);

\draw [fill=red] (5.98,0.2) circle (2.5pt);

\draw [fill=red] (7.28,0.2) circle (2.5pt);

\draw [fill=red] (6.04,2.2) circle (2.5pt);

\draw [fill=red] (3.82,0.2) circle (2.5pt);

\draw [fill=red] (1.8,1.2) circle (2.5pt);

\end{scriptsize}
\end{tikzpicture} &   \begin{tikzpicture}[scale=.5,line cap=round,line join=round,>=triangle 45,x=1.0cm,y=1.0cm]
\draw [line width=1.pt] (2.94,1.2)-- (4.42,1.2);
\draw [line width=1.pt] (4.42,1.2)-- (6.24,1.2);
\draw [line width=1.pt] (6.24,1.2)-- (7.24,2.2);
\draw [line width=1.pt] (6.24,1.2)-- (5.98,0.2);
\draw [line width=1.pt] (4.98,0.2)-- (5.98,0.2);
\draw [line width=1.pt] (4.42,1.2)-- (3.82,0.2);
\draw [line width=1.pt] (6.24,1.2)-- (7.28,0.2);
\draw [line width=1.pt] (2.94,1.2)-- (1.8,1.2);
\draw [line width=1.pt] (0.8,1.2)-- (1.8,1.2);
%\draw [line width=1.pt] (3.82,0.2)-- (2.8,0.2);
\begin{scriptsize}
\draw [fill=red] (2.94,1.2) circle (2.5pt);

\draw [fill=red] (4.42,1.2) circle (2.5pt);

\draw [fill=red] (6.24,1.2) circle (2.5pt);

\draw [fill=red] (7.24,2.2) circle (2.5pt);

\draw [fill=red] (5.98,0.2) circle (2.5pt);
%\draw [fill=red] (4.98,0.2) circle (2.5pt);

\draw [fill=red] (7.28,0.2) circle (2.5pt);

\draw [fill=red] (3.82,0.2) circle (2.5pt);

\draw [fill=red] (1.8,1.2) circle (2.5pt);
\draw [fill=red] (0.8,1.2) circle (2.5pt);
\draw[] (4.98,0.2) node {\Large\LeftScissors};
%\draw [fill=red] (2.8,0.2) circle (2.5pt);

\end{scriptsize}
\end{tikzpicture}\\
&& \begin{tikzpicture}[scale=.5,line cap=round,line join=round,>=triangle 45,x=1.0cm,y=1.0cm]
\draw [line width=1.pt] (2.94,1.2)-- (4.42,1.2);
\draw [line width=1.pt] (4.42,1.2)-- (6.24,1.2);
\draw [line width=1.pt] (6.24,1.2)-- (7.24,2.2);
\draw [line width=1.pt] (6.24,1.2)-- (5.98,0.2);
\draw [line width=1.pt] (4.42,1.2)-- (3.82,0.2);
\draw [line width=1.pt] (6.24,1.2)-- (7.28,0.2);
\draw [line width=1.pt] (2.94,1.2)-- (1.8,1.2);
\draw [line width=1.pt] (3.82,0.2)-- (2.8,0.2);
\begin{scriptsize}
\draw [fill=red] (2.94,1.2) circle (2.5pt);

\draw [fill=red] (4.42,1.2) circle (2.5pt);

\draw [fill=red] (6.24,1.2) circle (2.5pt);

\draw [fill=red] (7.24,2.2) circle (2.5pt);

\draw [fill=red] (5.98,0.2) circle (2.5pt);

\draw [fill=red] (7.28,0.2) circle (2.5pt);

\draw [fill=red] (3.82,0.2) circle (2.5pt);

\draw [fill=red] (1.8,1.2) circle (2.5pt);

\draw [fill=red] (2.8,0.2) circle (2.5pt);

\end{scriptsize}
\end{tikzpicture} & \begin{tikzpicture}[scale=.5,line cap=round,line join=round,>=triangle 45,x=1.0cm,y=1.0cm]
\draw [line width=1.pt] (2.94,1.2)-- (4.42,1.2);
\draw [line width=1.pt] (4.42,1.2)-- (6.24,1.2);
\draw [line width=1.pt] (6.24,1.2)-- (7.24,2.2);
\draw [line width=1.pt] (6.24,1.2)-- (5.98,0.2);
\draw [line width=1.pt] (4.98,0.2)-- (5.98,0.2);
\draw [line width=1.pt] (4.42,1.2)-- (3.82,0.2);
\draw [line width=1.pt] (6.24,1.2)-- (7.28,0.2);
\draw [line width=1.pt] (2.94,1.2)-- (1.8,1.2);
\draw [line width=1.pt] (3.82,0.2)-- (2.8,0.2);
\begin{scriptsize}
\draw [fill=red] (2.94,1.2) circle (2.5pt);

\draw [fill=red] (4.42,1.2) circle (2.5pt);

\draw [fill=red] (6.24,1.2) circle (2.5pt);

\draw [fill=red] (7.24,2.2) circle (2.5pt);

\draw [fill=red] (5.98,0.2) circle (2.5pt);
\draw [fill=red] (4.98,0.2) circle (2.5pt);

\draw [fill=red] (7.28,0.2) circle (2.5pt);

\draw [fill=red] (3.82,0.2) circle (2.5pt);

\draw [fill=red] (1.8,1.2) circle (2.5pt);

%\draw [fill=red] (2.8,0.2) circle (2.5pt);
\draw[] (2.8,0.2) node {\Large\LeftScissors};
\end{scriptsize}
\end{tikzpicture} \\
	\hline 
		&&&\\
&& \begin{tikzpicture}[scale=.5,line cap=round,line join=round,>=triangle 45,x=1.0cm,y=1.0cm]

\draw [line width=1.pt] (1.24,-1.4)-- (2.8,-1.4);
\draw [line width=1.pt] (2.8,-1.4)-- (3.92,-0.4);
\draw [line width=1.pt] (3.92,-0.4)-- (5.34,-0.4);
\draw [line width=1.pt] (5.34,-0.4)-- (6.28,-1.4);
\draw [line width=1.pt] (6.28,-1.4)-- (8.1,-1.4);
\draw [line width=1.pt] (1.18,0.8)-- (2.86,0.8);
\draw [line width=1.pt] (6.5,0.8)-- (8.12,0.8);
\draw [line width=1.pt] (2.86,0.8)-- (3.92,-0.4);
\draw [line width=1.pt] (6.5,0.8)-- (5.34,-0.4);
\begin{scriptsize}
\draw [fill=red] (1.24,-1.4) circle (2.5pt);
\draw [fill=red] (2.8,-1.4) circle (2.5pt);
\draw [fill=red] (3.92,-0.4) circle (2.5pt);
\draw [fill=red] (5.34,-0.4) circle (2.5pt);
\draw [fill=red] (6.28,-1.4) circle (2.5pt);
\draw [fill=red] (8.1,-1.4) circle (2.5pt);
\draw [fill=red] (1.18,0.8) circle (2.5pt);
\draw [fill=red] (2.86,0.8) circle (2.5pt);
\draw [fill=red] (6.5,0.8) circle (2.5pt);
\draw [fill=red] (8.12,0.8) circle (2.5pt);
\end{scriptsize}
\end{tikzpicture} &  \begin{tikzpicture}[scale=.5,line cap=round,line join=round,>=triangle 45,x=1.0cm,y=1.0cm]
\draw [line width=1.pt] (2.06,0.2)-- (3.58,0.2);
\draw [line width=1.pt] (2.06,0.2)-- (0.58,0.2);
\draw [line width=1.pt] (2.06,0.2)-- (2.06,1.0);
\draw [line width=1.pt] (2.06,0.2)-- (2.06,-0.8);
\begin{scriptsize}
\draw [fill=red] (2.06,0.2) circle (2.5pt);
\draw [fill=red] (3.58,0.2) circle (2.5pt);
\draw [fill=red] (0.58,0.2) circle (2.5pt);
\draw [fill=red] (2.06,-0.8) circle (2.5pt);
\draw[] (2.06,1.0) node {\Large\RightScissors};
\end{scriptsize}
\end{tikzpicture}\\ 
$x^2 -5x+1$	& $\frac{5\pm \sqrt{21}}{2}$ & \begin{tikzpicture}[scale=.5,line cap=round,line join=round,>=triangle 45,x=1.0cm,y=1.0cm]

\draw [line width=1.pt] (4.5,-0.0)-- (3.7,-0.0);
\draw [line width=1.pt] (4.48,0.8)-- (5.36,2.3);
\draw [line width=1.pt] (1.18,0.8)-- (2.86,0.8);
\draw [line width=1.pt] (6.5,0.8)-- (8.12,0.8);
\draw [line width=1.pt] (2.86,0.8)-- (3.7,-0.0);
\draw [line width=1.pt] (6.5,0.8)-- (5.38,-0.0);
\draw [line width=1.pt] (2.86,0.8)-- (2.88,2.3);
\draw [line width=1.pt] (6.5,0.8)-- (6.52,2.3);
\draw [line width=1.pt] (4.5,-0.0)-- (5.38,-0.0);
\draw [line width=1.pt] (4.48,0.8)-- (4.1,2.3);
\draw [line width=1.pt] (4.48,0.8)-- (4.5,-0.0);
\begin{scriptsize}
\draw [fill=red] (4.5,-0.0) circle (2.5pt);
\draw [fill=red] (3.7,-0.0) circle (2.5pt);
\draw [fill=red] (5.38,-0.0) circle (2.5pt);
\draw [fill=red] (4.48,0.8) circle (2.5pt);
\draw [fill=red] (5.36,2.3) circle (2.5pt);
\draw [fill=red] (1.18,0.8) circle (2.5pt);
\draw [fill=red] (2.86,0.8) circle (2.5pt);
\draw [fill=red] (6.5,0.8) circle (2.5pt);
\draw [fill=red] (8.12,0.8) circle (2.5pt);
\draw [fill=red] (2.88,2.3) circle (2.5pt);
\draw [fill=red] (6.52,2.3) circle (2.5pt);
\draw [fill=red] (4.1,2.3) circle (2.5pt);
\end{scriptsize}
\end{tikzpicture} & \begin{tikzpicture}[scale=.5,line cap=round,line join=round,>=triangle 45,x=1.0cm,y=1.0cm]
\draw [line width=1.pt] (2.94,1.2)-- (4.42,1.2);
\draw [line width=1.pt] (2.94,1.2)-- (2.94,0.2);
\draw [line width=1.pt] (4.42,1.2)-- (6.24,1.2);
\draw [line width=1.pt] (6.24,1.2)-- (7.24,2.2);
\draw [line width=1.pt] (4.42,1.2)-- (4.42,0.2);
\draw [line width=1.pt] (6.24,1.2)-- (7.28,0.2);
\draw [line width=1.pt] (2.94,1.2)-- (1.8,1.2);
\draw [line width=1.pt] (7.28,0.2)-- (8.66,0.2);
\draw [line width=1.pt] (7.24,2.2)-- (8.7,2.2);
\begin{scriptsize}
\draw [fill=red] (2.94,1.2) circle (2.5pt);
\draw [fill=red] (2.94,0.2) circle (2.5pt);
\draw [fill=red] (4.42,1.2) circle (2.5pt);
\draw [fill=red] (6.24,1.2) circle (2.5pt);
\draw [fill=red] (7.24,2.2) circle (2.5pt);
\draw [fill=red] (7.28,0.2) circle (2.5pt);
%\draw [fill=red] (3.82,0.2) circle (2.5pt);
\draw [fill=red] (1.8,1.2) circle (2.5pt);
\draw [fill=red] (8.66,0.2) circle (2.5pt);
\draw [fill=red] (8.7,2.2) circle (2.5pt);
\draw[] (4.42,0.2) node {\Large\RightScissors};
\end{scriptsize}
\end{tikzpicture}  \\
&& \begin{tikzpicture}[scale=.5,line cap=round,line join=round,>=triangle 45,x=1.0cm,y=1.0cm]

\draw [line width=1.pt] (4.86,3.2)-- (2.6,1.9);
\draw [line width=1.pt] (2.6,1.9)-- (2.6,0.6);
\draw [line width=1.pt] (4.86,3.2)-- (3.92,1.9);
\draw [line width=1.pt] (3.92,1.9)-- (3.92,0.6);
\draw [line width=1.pt] (4.9,3.2)-- (4.9,1.9);
\draw [line width=1.pt] (4.9,3.2)-- (7.28,1.9);
\draw [line width=1.pt] (7.28,1.9)-- (5.5,0.6);
\draw [line width=1.pt] (7.28,1.9)-- (6.5,0.6);
\draw [line width=1.pt] (7.28,1.9)-- (7.18,0.6);
\draw [line width=1.pt] (7.28,1.9)-- (8.08,0.6);
\draw [line width=1.pt] (7.28,1.9)-- (8.9,0.6);
\draw [line width=1.pt] (7.28,1.9)-- (9.9,0.6);
\begin{scriptsize}
\draw [fill=red] (4.9,3.2) circle (2.5pt);
\draw [fill=red] (2.6,1.9) circle (2.5pt);
\draw [fill=red] (2.6,0.6) circle (2.5pt);
\draw [fill=red] (3.92,1.9) circle (2.5pt);
\draw [fill=red] (3.92,0.6) circle (2.5pt);
\draw [fill=red] (4.9,1.9) circle (2.5pt);
\draw [fill=red] (7.28,1.9) circle (2.5pt);
\draw [fill=red] (5.5,0.6) circle (2.5pt);
\draw [fill=red] (6.5,0.6) circle (2.5pt);
\draw [fill=red] (7.18,0.6) circle (2.5pt);
\draw [fill=red] (8.08,0.6) circle (2.5pt);
\draw [fill=red] (8.9,0.6) circle (2.5pt);
\draw [fill=red] (9.9,0.6) circle (2.5pt);
\end{scriptsize}
\end{tikzpicture} & \begin{tikzpicture}[scale=.6,line cap=round,line join=round,>=triangle 45,x=1.0cm,y=1.0cm]
\draw [line width=1.pt] (7.,3.)-- (7.,2.);
\draw [line width=1.pt] (7.,3.)-- (6.,3.);
\draw [line width=1.pt] (7.,3.)-- (8.,3.);
\draw [line width=1.pt] (6.,3.)-- (6.,2.);
\draw [line width=1.pt] (7.,2.)-- (7.,1.);
\draw [line width=1.pt] (7.,1.)-- (6.,1.);
\draw [line width=1.pt] (7.,1.)-- (7.,0.);
\draw [line width=1.pt] (8.,3.)-- (9.,3.);
\draw [line width=1.pt] (8.,3.)-- (8.,2.);
\draw [line width=1.pt] (8.,2.)-- (8.,1.);
\draw [line width=1.pt] (8.,1.)-- (8.,0.);
\begin{scriptsize}
\draw [fill=red] (7.,3.) circle (2.5pt);
\draw [fill=red] (7.,2.) circle (2.5pt);
\draw [fill=red] (6.,3.) circle (2.5pt);
\draw [fill=red] (8.,3.) circle (2.5pt);
\draw [fill=red] (6.,2.) circle (2.5pt);
\draw [fill=red] (7.,1.) circle (2.5pt);
\draw [fill=red] (6.,1.) circle (2.5pt);
\draw [fill=red] (7.,0.) circle (2.5pt);
\draw [fill=red] (9.,3.) circle (2.5pt);
\draw [fill=red] (8.,2.) circle (2.5pt);
\draw [fill=red] (8.,1.) circle (2.5pt);
\draw[] (8,0) node {\Large\RightScissors};
\end{scriptsize}
\end{tikzpicture} \\
	\hline 
\end{tabular} 

	\caption{Some $\mu$-L-minimal $0,1$-cut-trees.}
\end{table}

\begin{table}[H]
	\begin{tabular}{|c|c|}
		\hline 
		$n$	& Tree \\ 
		\hline 
	$2$	& \begin{tikzpicture}[scale=.5,line cap=round,line join=round,>=triangle 45,x=1.0cm,y=1.0cm]
	\draw [line width=1.pt] (4.42,1.2)-- (6.24,1.2);
	\begin{scriptsize}
	\draw [fill=red] (4.42,1.2) circle (2.5pt);
	\draw [fill=red] (6.24,1.2) circle (2.5pt);
	\end{scriptsize}
	\end{tikzpicture} \\ 
		\hline 
	$10$	& \begin{tikzpicture}[scale=.5,line cap=round,line join=round,>=triangle 45,x=1.0cm,y=1.0cm]
	\draw [line width=1.pt] (4.28,1.4)-- (2.7,0.4);
	\draw [line width=1.pt] (4.28,1.4)-- (4.34,0.4);
	\draw [line width=1.pt] (4.28,1.4)-- (5.92,0.4);
	\draw [line width=1.pt] (2.7,0.4)-- (1.92,-0.4);
	\draw [line width=1.pt] (2.7,0.4)-- (3.22,-0.4);
	\draw [line width=1.pt] (4.34,0.4)-- (3.98,-0.4);
	\draw [line width=1.pt] (4.34,0.4)-- (4.96,-0.4);
	\draw [line width=1.pt] (5.92,0.4)-- (5.64,-0.4);
	\draw [line width=1.pt] (5.92,0.4)-- (6.6,-0.4);
	\begin{scriptsize}
	\draw [fill=red] (4.28,1.4) circle (2.5pt);
	\draw [fill=red] (2.7,0.4) circle (2.5pt);
	\draw [fill=red] (4.34,0.4) circle (2.5pt);
	\draw [fill=red] (5.92,0.4) circle (2.5pt);
	\draw [fill=red] (1.92,-0.4) circle (2.5pt);
	\draw [fill=red] (3.22,-0.4) circle (2.5pt);
	\draw [fill=red] (3.98,-0.4) circle (2.5pt);
	\draw [fill=red] (4.96,-0.4) circle (2.5pt);
	\draw [fill=red] (5.64,-0.4) circle (2.5pt);
	\draw [fill=red] (6.6,-0.4) circle (2.5pt);
	\end{scriptsize}
	\end{tikzpicture} \\ 
		\hline 
	$12$	& \begin{tikzpicture}[scale=.5,line cap=round,line join=round,>=triangle 45,x=1.0cm,y=1.0cm]
	\draw [line width=1.pt] (5.06,1.2)-- (4.34,0.2);
	\draw [line width=1.pt] (5.06,1.2)-- (5.92,0.2);
	\draw [line width=1.pt] (2.7,0.2)-- (1.92,-0.9);
	\draw [line width=1.pt] (2.7,0.2)-- (3.22,-0.9);
	\draw [line width=1.pt] (4.34,0.2)-- (3.98,-0.9);
	\draw [line width=1.pt] (4.34,0.2)-- (4.96,-0.9);
	\draw [line width=1.pt] (5.92,0.2)-- (5.64,-0.9);
	\draw [line width=1.pt] (5.92,0.2)-- (6.6,-0.9);
	\draw [line width=1.pt] (2.7,0.2)-- (3.28,1.2);
	\draw [line width=1.pt] (3.28,1.2)-- (4.22,1.9);
	\draw [line width=1.pt] (4.22,1.9)-- (5.06,1.2);
	\begin{scriptsize}
	\draw [fill=red] (5.06,1.2) circle (2.5pt);
	\draw [fill=red] (2.7,0.2) circle (2.5pt);
	\draw [fill=red] (4.34,0.2) circle (2.5pt);
	\draw [fill=red] (5.92,0.2) circle (2.5pt);
	\draw [fill=red] (1.92,-0.9) circle (2.5pt);
	\draw [fill=red] (3.22,-0.9) circle (2.5pt);
	\draw [fill=red] (3.98,-0.9) circle (2.5pt);
	\draw [fill=red] (4.96,-0.9) circle (2.5pt);
	\draw [fill=red] (5.64,-0.9) circle (2.5pt);
	\draw [fill=red] (6.6,-0.9) circle (2.5pt);
	\draw [fill=red] (3.28,1.2) circle (2.5pt);
	\draw [fill=red] (4.22,1.9) circle (2.5pt);
	\end{scriptsize}
	\end{tikzpicture} \\ 
		\hline 
		
	\end{tabular} 
	
	\caption{$2$-L-minimal $0$-cut-trees with at most $12$ vertices.}
\end{table}



%% References


\section*{Acknowledgements}
We would like to thank the unknown referees for their valuable comments. The first author is indebted to Iran National Science Foundation (INSF) for supporting this research under grant number 95005902. The second author is indebted to the School of Mathematics, Institute for Research in Fundamental
Sciences (IPM), Tehran, Iran for the support. The research of the second author was in part supported by grant from IPM (No. 99050211).

\bibliographystyle{model1-num-names}
\bibliography{<your-bib-database>}

\begin{thebibliography}{00}


\bibitem{BK} A. Bahmani, D. Kiani, Graph reduction techniques and the multiplicity of the Laplacian eigenvalues, Linear Algebra Appl., 503:215--232, 2016.

\bibitem{BK2} A. Bahmani, D. Kiani, On the multiplicity of the adjacency eigenvalues of graphs, Linear Algebra Appl., 477:1--20, 2015.

\bibitem{BH} A.E. Brouwer, W.H. Haemers, Spectra of Graphs, Universitext, Springer, New York, 2012.

\bibitem{CRS} D. Cvetkovic, P. Rowlinson, S.K. Simic, Eigenvalue bounds for the signless Laplacian, Publications de l'Institut Mathématique, 81(95):11-27, 2007.

\bibitem{FGWG} Y. Z. Fan, S. C. Gong, Y. Wang, Y.B. Gao, On trees with exactly one characteristic element, Linear Algebra Appl. 421.2-3, 233-242, 2007.

\bibitem{Fi} M. Fiedler, Eigenvectors of acyclic matrices, Czechoslovak Mathematical Journal 25.4, 607-618, 1975.


\bibitem{GL} L. L. Gardner, K. L. Krystina, The multiplicity of eigenvalues in the adjacency matrix of a tree, MIGHTY XXXV, Illinois State University, September 27 – 28, 2002.


\bibitem{GMc} C.D. Godsil, B.D. McKay, A new graph product and its spectrum, Bulletin of the Australian Mathematical Society, 18(1), 21-28, 1978.

\bibitem{GF} S.C. Gong, Y. Z. Fan, The property of maximal eigenvectors of trees, Linear and Multilinear Algebra 58.1 : 105-111, 2010.

\bibitem{GM} R. Grone, R. Merris, The Laplacian  spectrum of a graph II, SIAM Journal on discrete Mathematics 7, 221-229, 1994. 

\bibitem{GMS} R. Grone, R. Merris, V.S. Sunder, The Laplacian spectrum of a graph, SIAM J. Matrix Anal. Appl. 11, 218-238, 1990.


\bibitem{JLM} C.R. Johnson, A. Leal Duarte, P.R. McMichael, Changes in vertex status and the fundamental decomposition of a tree relative to a multiple (Parter) eigenvalue, Disc. App. Math. 2017.

\bibitem{JLS}  C. R. Johnson,  A. Leal Duarte,  C. M. Saiago, The Parter-Wiener theorem:  refinement and generalization, SIAM Journal of Matrix Analysis and Applications 25:  311–330, 2003.

\bibitem{JLSS}C. R. Johnson,  A. Leal Duarte,  C. M. Saiago, D. Sher, Eigenvalues, multiplicities and graphs, Contemporary Mathematics 419 : 167, 2006.



\bibitem{MS} K. H. Monfared, B. L. Shader, The nowhere-zero eigenbasis problem for a graph, Linear Algebra Appl. 505 : 296-312, 2016.

\bibitem{Pr} V.V. Prasolov, Problems and theorems in linear algebra, Vol. 134. American Mathematical Soc., 1994.
\bibitem{S1} J. Salez, Every totally real algebraic integer is a tree eigenvalue, Journal of Combinatorial Theory, Series B 111 : 249-256, 2015.

\bibitem{S2} J. Salez, Spectral atoms of unimodular random trees, <hal-01374519>. arXiv preprint arXiv:1609.09374, 2016.



\end{thebibliography}



\end{document}